\documentclass{article}

\usepackage{times}
\usepackage[pdftex]{graphicx}
\usepackage{amssymb,amsfonts,amsmath,amsthm}
\usepackage{float}
\usepackage{xcolor}
\usepackage{enumitem}

\newtheorem{theorem}{Theorem}
\newtheorem{proposition}{Proposition}
\newtheorem{lemma}{Lemma}
\newtheorem{ass}{Assumption}

\newtheorem{remark}{Remark}
\newtheorem{cor}{Corollary}

\newcommand\EE {\mathbb E}

\newcommand\NN {\mathbb N}
\newcommand\RR {\mathbb R}
\newcommand\PP {\mathbb P}

\newcommand\ZZ {\mathbb Z}

\def\bone{\mathbf{1}}

\def\qed{\hskip6pt\vrule height6pt width5pt depth1pt}

\def\qed{\hskip 6pt\vrule height6pt width5pt depth1pt}

\newcommand{\ed}{\end{document}}
\newcommand{\be}{\begin{equation}}
\newcommand{\ee}{\end{equation}}
\newcommand{\bq}{\begin{eqnarray}}
\newcommand{\eq}{\end{eqnarray}}

\vspace{4in}

\newcommand{\mmod}{\mathrm{mod}}

\definecolor{Red}{rgb}{0.9,0,0.0}
\definecolor{Blue}{rgb}{0,0.0,1.0}

\begin{document}

\title{Consistency of MLE for partially observed diffusions, with application in market microstructure modeling}
\author{Sergey~Nadtochiy, and Yuan~Yin\footnote{The authors thank I. Cialenco for useful discussions and S. Jaimungal for providing the historical data used herein.}
\footnote{Partial support from the NSF CAREER grant 1855309 is acknowledged.}
\footnote{The data that support the findings of this study are available from the corresponding author upon reasonable request.}
\footnote{Address the correspondence to: Department of Applied Mathematics, Illinois Institute of Technology, 10 W. 32nd St., Chicago, IL 60616 (snadtochiy@iit.edu).}
}
\date{\today}

\maketitle

\begin{abstract}
This paper presents a tractable sufficient condition for the consistency of maximum likelihood estimators (MLEs) in partially observed diffusion models, stated in terms of stationary distribution of the associated fully observed diffusion, under the assumption that the set of unknown parameter values is finite. This sufficient condition is then verified in the context of a latent price model of market microstructure, yielding consistency of maximum likelihood estimators of the unknown parameters in this model. Finally, we compute the latter estimators using historical financial data taken from the NASDAQ exchange.
\end{abstract}



\section{Introduction}

This paper is concerned with the large-sample properties of maximum likelihood estimators (MLEs) for partially observed diffusions, as well as their application in market microstructure modeling. 
The need for such estimation arises when, in addition to the unobserved dynamic components, a partially observed dynamical system has parameters that are assumed to be constant but not known.
A fairly general form of a partially observed diffusion with unknown parameters is given by \eqref{eq.Background.dyn.X}--\eqref{eq.Background.dyn.Y}, and a more specific model, motivated by financial application, is discussed in Section \ref{se:latentprice}. 

The choice of MLE as an estimator is justified by its large-sample properties. It is well known (see, e.g., \cite{Engle}) that the MLE is consistent (i.e., it converges to the true parameter value as the sample size grows) and asymptotically efficient (i.e., its asymptotic variance attains the famous lower bound due to Cramer-Rao) in the cases where the observed sample consists of independent identically distributed (i.i.d.) random variables. In this sense, MLE has the best large-sample properties. Some of these asymptotic properties of MLE have been extended to the settings where the observations are not i.i.d. but are given by a sample path of a Markov process: see \cite{BagchiFullObs}, \cite{BagchiFullObs2}, and the references therein, for the case of fully observed diffusions with linear coefficients.
A general approach for analyzing the asymptotic properties of MLE (as well as of other statistical methods) is described, e.g., in \cite{IbrahimovKhasminskii} and \cite{KutoyantsBook}. For the partially observed discrete-time Markov chains, there exist general results on the consistency of MLE (see, e.g., \cite{RamonMLEPartObsDisc}, \cite{LaredoDisc}, \cite{BookMLEpartObsDisc}), but it is not straightforward to extend the discrete-time methods and results to continuous-time processes. To date, with the exception of the present work and to the best of our knowledge, there exists only one general result on the (large-sample) consistency of MLE for partially observed diffusions, stated in \cite{ChiganskiMLE}.

The above literature review is focused on the large-sample properties of MLE for partially observed Markov processes, as such is the focus of the present paper. However, it is worth mentioning the related (though different) results on the asymptotic properties of MLEs for Markov systems with vanishing noise (cf. \cite{KutoyantsBookSmallNoise}, \cite{KutoyantsSpecific}, \cite{LeGland}, and the references therein), on the large-sample properties of the so-called one-step MLE process (cf. \cite{LeCam}, \cite{KhasminskiKutoyantsSpecific}, and the references therein), and on the EM algorithm that can be used to efficiently compute MLEs (cf. \cite{DemboZeitouni}).

The analysis of MLE for partially observed Markov systems is intimately connected to the problem of stochastic filtering: i.e., the computation of conditional distribution of the unobserved components given the information available thus far -- this conditional distribution is referred to as the filter (see, e.g., \cite{BainCrisan}, and the references therein).
Indeed, one may express the likelihood function and MLE via the filter and the observed component and notice that the latter form a new (fully observed) Markov process. Then, one may try to apply the abstract results (see, e.g., \cite{IbrahimovKhasminskii}, \cite{KutoyantsBook}) to establish the desired consistency and asymptotic efficiency of MLE in this fully observed Markov system. This is precisely the approach taken in \cite{ChiganskiMLE}, which provides sufficient conditions for the consistency and asymptotic efficiency of MLE in partially observed finite-state continuous-time Markov models. 
However, with very few exceptions (e.g., Kalman-Bucy filter), the filter in a general nonlinear partially observed Markov system is not available in a closed form. Moreover, the joint dynamics of the (typically, infinite-dimensional) filter and the observation are convoluted and degenerate.
For this reason, the aforementioned approach based on the classical abstract results leads to assumptions that are difficult to verify in practice: e.g., stated in terms of the quantitative properties of the likelihood function or of the stationary distribution of the filter, for which there is no tractable characterization in general. The latter motivates Theorem \ref{thm:NY} whose assumptions are stated in terms of stationary distribution of the original (finite-dimensional) diffusion, which has a tractable characterization, e.g., via the associated elliptic partial differential equation (PDE).

The main theoretical contribution of this work is two-fold. First, we derive a tractable sufficient condition for the consistency of MLE for partially observed diffusions (Theorem \ref{thm:NY}, Section \ref{se:general}). Second, we verify this sufficient condition for a specific model motivated by a financial application (Theorem \ref{thm:main}, Section \ref{se:latentprice}). The sufficient condition of Theorem \ref{thm:NY} requires one to find an appropriate set of separating test functions and test processes, in order to distinguish between any two values of the unknown parameter. The main technical effort of this paper is in the proof of Theorem \ref{thm:main}, where we illustrate how to choose such test functions and test processes, and how to verify that they possess the desired separating property, in the context of a latent price model of market microstructure. This verification is based on a careful asymptotic analysis of certain functionals of the associated diffusion (see Proposition \ref{prop:LatentPrice.verify}). We believe that the methods presented in the proof of Theorem \ref{thm:main} can be applied to other models of interest, and we discuss the general approach to verifying the conditions of Theorem \ref{thm:NY} in the discussion following the latter theorem. It is worth mentioning that Theorem \ref{thm:NY} requires the set of the unknown parameter values to be finite. A relaxation of this assumption (via an appropriate continuity assumption on the diffusion coefficients) is the subject of ongoing research in \cite{EkrenNadtochiy}, which uses the finite-parameter result proved herein as an important building block.

The theoretical results developed in the present paper are motivated by a specific financial application. It is a notorious problem in market microstructure modeling to find the right notion of price for a given asset, that would be appropriate in a setting where the observations are made at a very high (time) frequency. For example, the quoted bid and ask prices of even the most liquid stocks change only once in several seconds, whereas other relevant events in the limit order book occur on a scale of milliseconds. Thus, any notion of a price based on the observed bid and ask quotes (e.g., the midprice, or the last transaction price) would remain constant most of the time and would not reflect with sufficient accuracy the consensus of market participants about the true value of the asset. As a result, it is very challenging to define and estimate the volatility of midprice (or of the last transaction price) in a high-frequency regime. The fact that the changes of midprice and the changes of other relevant market factors occur on two very different time scales also makes it difficult to build a predictive model for the future price. In particular, it is challenging to estimate the price impact (i.e., the expected change in the price caused by a purchase or a sale of one unit of the asset) using high frequency observations. Indeed, it takes a relatively large number of trades to move the quoted prices (which stay constant most of the time), hence the standard linear regression model for the two does not work well in the high frequency regime. This issue is documented in \cite{Cont}, \cite{Livieri}, whose empirical results indicate that the estimation of expected impact of trades on midprice only works well if the observations are taken at a fairly low frequency (tens of seconds), which means that a large fraction of observations have to be ignored in such analysis.
To address these challenges, and to obtain a ``more dynamic" notion of price, the latent price models are used. Namely, instead of modeling the quoted prices directly, one assumes the existence of an unobserved (latent) price process, which changes at a high frequency, such that the quoted (observed) prices are functions (e.g., roundings to the nearest integer number of cents) of the latent price: see, e.g., \cite{Rosenbaum}, and the references therein. The advantage of a latent price model is that the price and other relevant market factors evolve on the same time scale. The main disadvantage is that the true price is hidden, which leads to a system with partial observations.
The latent price model analyzed herein is a version of the model proposed in \cite{N} to explain the concavity of expected price impact of a large sequence of co-directed trades -- a phenomenon that has received much attention from both practitioners and academics (see the discussion and references in \cite{N}).
This model is a mixture of the uncertainty-zone model of \cite{Rosenbaum} and a linear price impact model (see, e.g., \cite{AlmgrenChriss}), but it also prescribes an additional micro-drift for the latent price process (which is responsible for the aforementioned concavity of price impact). A more detailed description of the proposed model is presented at the beginning of Section \ref{se:latentprice}, which also contains the proof of consistency of MLE in this model (via verification of the conditions of Theorem \ref{thm:NY}). The empirical analysis described in Section \ref{se:implem} shows how to compute MLE in the proposed latent price model numerically, using historical financial data. The results of this analysis, in particular, provide high-frequency estimates of the price impact coefficient and of the volatility of latent price. In addition, these empirical results support the presence of the aforementioned micro-drift in the latent price process (i.e., the estimated magnitude of the micro-drift is non-negligible), which validates the main novel feature of the proposed model.

\section{Consistency of MLE for partially observed diffusions}
\label{se:general}

Consider a diffusion process $(X,Y)$ evolving in $\RR^{q+m}$ according to
\begin{align}
& dX_t = b(\theta;X_t) dt + \sigma(\theta;X_t) dB_t,\label{eq.Background.dyn.X}\\
& dY_t = h(\theta;X_t) dt + \bar\sigma dW_t,\label{eq.Background.dyn.Y}
\end{align}
where $(B,W)$ is a $\RR^{d+m}$-valued Brownian motion, with possibly correlated components: namely, $(B,W)$ is a continuous process with stationary, independent and centered increments, and the components of $W$ are independent.\footnote{The large-sample consistency of MLE, established in Theorem \ref{thm:NY}, holds for more general (even non-Markovian) processes $X$, as long they satisfy an appropriate ergodicity property.} We assume that the initial condition $X_0$ is independent of $(B,W)$, that the unknown parameter $\theta$ belongs to a given set $\Theta$, that the matrix $\bar\sigma\in \RR^{m\times m}$ is invertible, and that $b(\theta;\cdot),h(\theta;\cdot),\sigma(\theta;\cdot)$ are Borel functions (vector- and matrix-valued, respectively). W.l.o.g., we assume that $Y_0=0$. We also assume that there exists a weak solution to \eqref{eq.Background.dyn.X}--\eqref{eq.Background.dyn.Y} that is unique in law.

The component $Y$ is observed while $X$ is not. Our goal is to estimate the true value of the unknown parameter $\theta$ using the observation $Y_{[0,T]}$. More specifically, we are interested in MLE due to its expected efficiency as $T\rightarrow\infty$ (i.e., as the sample size grows). MLE is constructed by choosing a reference measure $\PP_T$ that dominates the law of $Y$ on the canonical space $C([0,T],\RR^m)$, denoted $\PP^\theta_T$, and by computing the \emph{likelihood} function which is defined as the Radon-Nikodym derivative $d\PP^\theta_T/d\PP_T$ evaluated on the observed path $Y_{[0,T]}$.\footnote{If one chooses two different dominating measures, so that one of them also dominates the other, the MLE under both measures remains the same.}
Under natural integrability assumptions on $h(\theta;\cdot)$ (e.g., if $h(\theta;\cdot)$ is bounded), Girsanov's theorem implies that $\PP^\theta_T$ is equivalent to the distribution of $\bar\sigma W$ on $C([0,T],\RR^m)$, denoted $\PP_T$.\footnote{If $W$ is a standard Brownian motion in $\RR^m$, then $\PP_T$ is a scaled Wiener measure.} Using the classical results on innovation processes (see, e.g., \cite{LipsterShiryaev}), one easily derives the following formula for the likelihood: for any $\theta\in\Theta$,
\begin{align}\label{eq.Background.L.first.def}
L_T(\theta):=\frac{d\PP^\theta_T}{d\PP_T} = \exp\left(-\frac{1}{2}\int_0^T \left|\bar\sigma^{-1}h^{\theta}_t\right|^2 dt
+ \int_0^T \left(\bar\sigma^{-1}h^{\theta}_t\right)^{\top}\, \bar\sigma^{-1}dY_t \right),
\quad h^\theta_t:=\tilde\EE^{\theta}\left(h(\theta;X_t)\,\vert\,\mathcal{F}^Y_t\right),
\end{align}
where $\tilde\EE^\theta$ is the expectation under $\tilde\PP^\theta$, the law of $(X,Y)$ on $C([0,\infty),\RR^{q+m})$ assuming the parameter value $\theta$ in \eqref{eq.Background.dyn.X}--\eqref{eq.Background.dyn.Y}. Note that $\PP^\theta$, the law of $Y$ on $C([0,\infty),\RR^{m})$, is the second marginal of $\tilde\PP^\theta$.

MLE is defined naturally as
\begin{align*}
&\hat{\theta}_T \in \text{argmax}_{\theta\in\Theta} L_T(\theta).
\end{align*}

As discussed in the introduction, MLE is known to possess the best large-sample properties in the case of i.i.d. observations. However, much less is known about its asymptotic properties in the present setting. The following theorem provides a tractable sufficient condition for the consistency of MLE in the above setting of partially observed diffusion model.

\begin{theorem}\label{thm:NY}
Assume that $\Theta$ is finite and that each $h(\theta;\cdot)$ is absolutely bounded. Assume also that there exist an integer $k\geq1$, a set of indices $\Lambda$, and a family of test processes $\{Y^\lambda\}_{\lambda\in\Lambda}$, each taking values in $\RR^{k}$ and adapted to $\mathbb{F}^Y$, s.t.
\begin{itemize}
\item $(h(X_t),Y^\lambda_t)$ is strongly ergodic under each $\tilde \PP^\theta$: i.e., for each $\theta\in\Theta$ there exists a measure $\nu^{\theta,\lambda}$ on $\RR^{m}\times\RR^{k}$ s.t.
\begin{align*}
\int_{\RR^{m}\times\RR^k}\phi(z,y) \nu^{\theta,\lambda}(dz,dy) = \lim_{T\rightarrow\infty} \frac{1}{T} \int_0^T \phi(h(X_t),Y^\lambda_t) dt,\quad \tilde\PP^\theta\text{-a.s.},
\end{align*} 
for any bounded Borel $\phi:\RR^{m}\times\RR^k\rightarrow\RR $;
\item for any $\theta,\theta'\in\Theta$, the equality
\begin{align}\label{eq.Background.nu.sep}
\int_{\RR^{m}\times\RR^k} z\cdot f(y)\, \nu^{\theta,\lambda}(dz,dy)=\int_{\RR^{m}\times\RR^k} z\cdot f(y)\, \nu^{\theta',\lambda}(dz,dy)
\quad \forall \,\lambda\in\Lambda\, \text{and bounded Borel }f:\RR^k\rightarrow\RR^m
\end{align}
implies $\theta=\theta'$.
\end{itemize}
Then, for any $\theta\in\Theta$, as $T\rightarrow\infty$, any MLE $\hat\theta_T$ converges to $\theta$ in probability (and a.s.) under $\PP^\theta$.
\end{theorem}


\smallskip

Before we prove Theorem \ref{thm:NY}, let us comment on the practical value of this result. First, it is clear that the consistency of MLE implies the identifiability of $\theta$: i.e., for any $\theta\neq\theta'$, the two measures $\PP^{\theta}$ and $\PP^{\theta'}$ are mutually singular. The latter is interpreted as the feasibility of the problem of estimating the true value of $\theta$ with certainty from a sample with infinite time horizon. If the set $\Theta$ is finite, it is not hard to show that the identifiability of $\theta$ yields Proposition \ref{prop:main} and in turn the consistency of MLE. Thus, the statement of Theorem \ref{thm:NY} is equivalent to saying that, under the assumptions of this theorem, $\theta$ is identifiable. Moreover, since the notion of identifiability does not depend on whether $\Theta$ is finite or not, we conclude that the conditions of Theorem \ref{thm:NY}, without the assumption that $\Theta$ is finite, yield the identifiability of $\theta$. 

On the one hand, the assumption of finite $\Theta$ may not be too restrictive in practice, in the sense that any numerical algorithm for the computation of MLE is automatically restricted to take a value in a finite set (e.g., determined by the machine precision). On the other hand, in order to give a practical meaning to the statement that the MLE $\hat\theta_t$ converges to $\theta$ in probability under $\PP^\theta$, one needs to assume that the ``true" distribution of the observation corresponds to one of the values of $\theta$ restricted to a chosen finite set. The latter may or may not be natural, depending on a specific model/application. In general, it is certainly valuable to extend Theorem \ref{thm:NY} to a setting with a general (in particular, infinite) set $\Theta$. Such an extension is the subject of the ongoing investigation in \cite{EkrenNadtochiy}, which aims to prove that the (abstract) identifiability implies the consistency of MLE in partially observed diffusion models on a torus, for a general parameter set $\Theta$. 

\smallskip

The verification of the conditions of Theorem \ref{thm:NY} may require some work, but it is not unrealistic as illustrated by the present discussion and in Section \ref{se:pf.Thm.main}. First, one needs to choose a family of test processes $\{Y^\lambda\}_{\lambda\in\Lambda}$, adapted to $\mathbb{F}^Y$, s.t. $(h(X),Y^\lambda)$ is strongly ergodic. Note that $h(X)$ is ergodic as long as $h(\theta,\cdot)$ is bounded and either $X$ is ergodic or $(b,\sigma,h)$ are periodic and $\sigma\sigma^\top$ is strictly positive definite. Then, one can ensure the ergodicity of $(h(X),Y^\lambda)$ by choosing $Y^\lambda$ so that it forms a Markov system with $X$ and is strongly recurrent. We propose to define $Y^\lambda$ as follows:
\begin{align}\label{eq.sec2.Ylambda.def1}
dY^{\lambda}_t = -\lambda Y^{\lambda}_t dt + dY_t,\quad Y^{\lambda}_0=0,
\quad Y^\lambda_t = \int_0^t e^{-\lambda(t-s)}\,dY_s.
\end{align}
for all $\lambda\in(0,\infty)=:\Lambda$, as it is done in Section \ref{se:pf.Thm.main}.

The main challenge of using Theorem \ref{thm:NY} is in verifying its second condition, \eqref{eq.Background.nu.sep}. 
Notice, however, that the condition \eqref{eq.Background.nu.sep} is stated in terms of the invariant measure of the {\bf fully observed} system $(h(X),Y^\lambda)$, which is a probability measure on a finite-dimensional state space: e.g., the density of $\nu^{\theta,\lambda}$ is a solution to an elliptic PDE, whose coefficients are given explicitly in terms of $h,b,\bar\sigma,\sigma$. One may use the aforementioned PDE to obtain estimates or numerical approximations of the integrals in \eqref{eq.Background.nu.sep} and to verify the second condition of Theorem \ref{thm:NY}, provided the set $\Theta$ is not too large. In addition, one may combine the probabilistic and PDE tools to verify  \eqref{eq.Background.nu.sep} without explicitly computing, or approximating numerically, the quantities involved in \eqref{eq.Background.nu.sep}. Below, we consider two examples of families of partially observed diffusion models in which such a verification is possible (in both examples, $X$ is assumed to be ergodic, although, if $(b,\sigma,h)$ are periodic and $\sigma\sigma^\top$ is strictly positive definite, one may instead use the ergodicity of the canonical projection of $X$ on the associated torus).

\smallskip

In the first example, we assume that $\theta$ is a vector in $\RR^m$, that $(b,\sigma)$ do not depend on $\theta$, and that every $h_i$ depends on $\theta$ only through $\theta_i$. Then, we do not need to utilize $Y^\lambda$, and we show that, under certain regularity and monotonicity assumptions, the equation \eqref{eq.Background.nu.sep} with $f\equiv1$ yields $\theta=\theta'$. To this end, we notice that the left hand side of \eqref{eq.Background.nu.sep} with $f\equiv1$ can be written in terms of the invariant density $p$ of $X$ under $\tilde\PP^\theta$ as $\int_{\RR^q}h(\theta,\cdot)p\,dx$ (note that $p$ does not depend on $\theta$). Then, assuming that, for every $i$, the function $\theta_i\mapsto h_i(\theta_i,\cdot)$ is differentiable\footnote{Although $\Theta$ is finite, in this example, we treat $h_i(\theta_i,\cdot)$ as a function of $\theta_i\in\RR$.}, that its derivative is uniformly integrable in $x$ (w.r.t. $p\,dx$) and uniformly continuous in $\theta_i$, and that the function $x\mapsto\partial_{\theta_i}h(\theta_i,x)$ is either strictly positive or strictly negative, we deduce that the Jacobian of the $C^1$ mapping $\theta\mapsto \int_{\RR^q}h(\theta,\cdot)p\,dx$ is invertible at each $\theta$. Since $\Theta$ is finite, we can restrict $\theta$ to a compact set and conclude that the mapping $\theta\mapsto \int_{\RR^q}h(\theta,\cdot)p\,dx$ is bijective. The latter means that \eqref{eq.Background.nu.sep} with $f\equiv1$ yields $\theta=\theta'$.

\smallskip

In the second example, we assume that $X$ that $W$ is uncorrelated with $B$\footnote{Note that this assumption is not satisfied in the model considered in Section \ref{se:latentprice}, which is the main reason why its analysis is more complicated.}, that $Y$ is one-dimensional, 
and that $\bar\sigma=1$. Then, the left hand side of \eqref{eq.Background.nu.sep}, with $Y^\lambda$ given by \eqref{eq.sec2.Ylambda.def1} and with $f(y)=y^j$, is equal to 
\begin{align*}
& \lim_{t\rightarrow\infty}\tilde\EE^\theta \left[h(\theta,X_t) (Y^\lambda_t)^j\right]
= \sum_{i=0}^j \lim_{t\rightarrow\infty}\tilde\EE^\theta \left[h(\theta,X_t)\, (H_t(\theta))^i\right]\,\tilde\EE K_t^{j-i},
\end{align*}
where
\begin{align*}
& H_t(\theta):=\int_0^t e^{-\lambda(t-s)}\,h(\theta,X_s)\,ds,\,\,K_t:=\int_0^t e^{-\lambda(t-s)}\,dW_s,
\quad \lim_{\lambda\rightarrow\infty}\lambda H_t(\theta)=h(\theta,X_t).
\end{align*}
Assuming that one can justify the interchanging of the above limits in $t$ and $\lambda$ (and interchange both of them with $\tilde\EE^\theta$), one can deduce from the above (used recursively, for $j=0,1,\ldots$) that the equality in \eqref{eq.Background.nu.sep} yields
\begin{align*}
\lim_{t\rightarrow\infty}\EE^\theta (h(\theta,X_t))^j = \lim_{t\rightarrow\infty}\EE^{\theta'} (h(\theta',X_t))^j,\quad j=0,1,\ldots
\end{align*}
Next, we recall that, in many cases, the sequence of all moments determines a distribution uniquely. Once the above program is completed (Section \ref{se:pf.Thm.main} shows how to implement this strategy in an even more complicated setting where $W$ is correlated with $B$), we conclude that, in order to verify the second condition of Theorem \ref{thm:NY}, it suffices to establish the following: the invariant distribution of $h(\theta,X)$ under $\tilde\PP^\theta$ determines $\theta$ uniquely. The latter is a natural and tractable condition: in fact, this is a typical sufficient condition for the identifiability of $\theta$ in the fully observed diffusion models (and, clearly, a partially observed model is not expected to have more tractable sufficient conditions for identifiability than the associated fully observed model).
Let us show how the claim ``invariant distribution of $h(\theta,X)$ under $\tilde\PP^\theta$ determines $\theta$ uniquely" can be verified in certain classes of diffusions models. Namely, we assume that $h$ does not depend on $\theta$ and that the mapping $x\mapsto h(x)$ is invertible. Then, the desired claim reduces to the statement that the invariant distribution of $X$ under $\tilde\PP^\theta$ determines $\theta$ uniquely. To verify the latter, we recall that the invariant density $p^\theta$ of $X$ under $\tilde\PP^\theta$ satisfies the elliptic PDE
\begin{align*}
& \frac{1}{2}\sum_{i,j=1}^q \partial_{x_ix_j}\left[a_{ij}p^\theta\right] - \sum_{i=1}^q\partial_{x_i}\left[b_{i}p^\theta\right]=0,
\quad a:=\sigma\sigma^\top.
\end{align*}
Arguing by contradiction and assuming that there exist $\theta\neq\theta'$ with $p^\theta=p^{\theta'}=:p$, we obtain
\begin{align*}
& \frac{1}{2}\sum_{i,j=1}^q \partial_{x_ix_j}\left[(a_{ij}(\theta,\cdot)-a_{ij}(\theta',\cdot))p\right] - \sum_{i=1}^q\partial_{x_i}\left[(b_{i}(\theta,\cdot)-b_i(\theta',\cdot))p\right]=0.
\end{align*}
There exist various ways to deduce a contradiction with the above, in specific models in which the identifiability is natural to expect. To provide a simple explicit example, we assume that $q=1$, that $\theta\in\RR$, and that the drift $b$ has the property that, for any $\theta\neq\theta'$, the function $x\mapsto b(\theta,x)-b_i(\theta',x)$ is strictly positive or strictly negative. Then, multiplying the above display by $x$ and integrating, we obtain $\int_{\RR}(b(\theta,\cdot)-b(\theta',\cdot))\,p\,dx=0$, which yields a desired contradiction.

\smallskip

\begin{remark}\label{rem:compareToDiscrete}
Note that the continuous-time setting used herein, in fact, does not complicate the analysis as compared to the discrete-time setting. On contrary, it turns out that the diffusion models naturally lead to a more convenient choice of a reference measure with respect to which the likelihood is computed (i.e., a measure with respect to which the density of the observation is computed). The existing results on discrete-time Markov processes (see the introduction for a list of such results) express the sufficient condition for consistency in terms of the transition kernel of the Markov process, which naturally leads to the (product) Lebesgue measure as a reference measure for the likelihood. We believe that the sufficient condition of Theorem \ref{thm:NY} is more tractable than the existing conditions for discrete-time Markov processes (and the arguments in our proof are simpler), because the continuous-time framework leads us naturally to the (scaled) Wiener measure as a reference measure for the likelihood, which turns out to be more convenient for analyzing the consistency of MLE.
\end{remark}


\medskip

In the remainder of this section, we prove Theorem \ref{thm:NY}.
It is well known (see, e.g., \cite{IbrahimovKhasminskii}) that the consistency of $\hat\theta$ can be deduced from the following proposition.

\begin{proposition}\label{prop:main}
Under the assumptions of the theorem, for any distinct $\theta,\theta'\in\Theta$, we have $\lim_{T\rightarrow\infty}L_T(\theta')/L_T(\theta)= 0$, $\PP^\theta$-a.s..
\end{proposition}
\begin{proof}
To ease the notation, we assume that $\bar\sigma$ is the identity matrix.
A direct computation shows
\begin{equation}\label{eq.consist.compute M}
M_T:=\frac{L_T(\theta')}{L_T(\theta)}
= \exp\left(-\frac{1}{2}\int_0^T \left|\Delta h_t^{\theta',\theta}\right|^2 dt
+ \int_0^T \Delta h_t^{\theta',\theta} d W^{\theta}_t \right),
\end{equation}
where $W^\theta$ is a $\PP^\theta$-Brownian motion and $\Delta h_t^{\theta',\theta} := h^{\theta'}_t-h^{\theta}_t$.
Thus, $M$ is a positive supermartingale under $\PP^\theta$, which therefore converges a.s. under $\PP^\theta$ to a limit denoted by $\xi$. Our goal is to show that $\xi=0$, $\PP^\theta$-a.s..

\smallskip

We argue by contradiction and assume that $\PP^\theta(\xi>0)>0$.
Then, we claim that there exists $\epsilon>0$ s.t. $\PP^\theta(\xi\geq\epsilon)>0$ and
\begin{equation}\label{eq.MLEConsistPf.Ind.Conv1}
\lim_{T\to\infty}\bone_{\{M_T\ge\epsilon\}} = \bone_{\{\xi\geq\epsilon\}},
\end{equation}
a.s. under $\PP^\theta$.
The first part of this clam is obvious and the second follows from the observation that the set of mass points of the distribution of $\xi$ is at most countable, and from the application of the continuous mapping theorem.
Moreover, noticing that $M^{-1}$ is a positive supermartingale under $\PP^{\theta'}$, we conclude that there exist $\epsilon>0$ and $\eta$ such that
\begin{equation}\label{eq.MLEConsistPf.Ind.Conv2}
\lim_{T\to\infty} M^{-1}_T = \eta,\quad
\lim_{T\to\infty}\bone_{\{M_T\ge\epsilon\}} = \bone_{\{\eta\leq\epsilon^{-1}\}},
\end{equation}
a.s. under $\PP^{\theta'}$.

\smallskip

Then, with the notation $A:=\{\xi\ge\epsilon\}$, we claim that
\begin{equation}\label{eq.consist.concequence1}
\lim_{T\to\infty}\frac{1}{T}\int_0^T \left|\Delta h_t^{\theta',\theta}\right|^2 dt = 0 \quad\PP^\theta\mbox{-a.s. on $A$}
\end{equation}
To prove \eqref{eq.consist.concequence1}, we notice that, if it doesn't hold, there exists an event $A'\subset A$, s.t. $\PP^\theta(A')>0$ and
$$
\liminf_{T\to\infty}-\frac{1}{T}\int_0^{T} \left|\Delta h_t^{\theta',\theta}\right|^2 dt < 0 \quad\mbox{on $A'$}
$$
Then, on $A'$, we have
\begin{align*}
&\liminf_{T\to\infty} M_T
= \liminf_{T\to\infty} \exp\left(\frac{T}{2}\left[-\frac{1}{T}\int_0^T \left|\Delta h_t^{\theta',\theta}\right|^2 dt
+ \frac{2}{T}\int_0^T \Delta h_t^{\theta',\theta} d W^{\theta}_t\right] \right) = 0,
\end{align*}
as the term $\frac{1}{T}\int_0^T \Delta h_t^{\theta',\theta} d W^{\theta}_t$ converges to zero in probability (under $\PP^\theta$) and hence a.s. along a subsequence of any sequence of $\{T_n\}$.
The above is a contradiction as $A'\subset A$ and $\PP^\theta(A')>0$. Thus, we conclude that \eqref{eq.consist.concequence1} holds.

\smallskip

Next, using the dominated convergence theorem, we deduce from \eqref{eq.consist.concequence1} that
\begin{equation*}
\lim_{T\to\infty}\EE^\theta \frac{1}{T}\int_0^T \left|\Delta h_t^{\theta',\theta}\right|^2 dt\, \bone_A = 0.
\end{equation*}
The above, along with the Cauchy-Schwartz inequality, yields the following for any bounded test function $f:\RR^k\rightarrow\RR^m$ and any test process $Y^\lambda$ satisfying the assumptions of the theorem:
\begin{align*}
0 &\le \lim_{T\to\infty} \EE^\theta\frac{1}{T}\int_0^{T} \bone_{A}\, \left|f\left(Y^\lambda_t\right)\cdot \Delta h_t^{\theta',\theta}\right| dt \\
&\le \lim_{T\to\infty} \left(\EE^\theta \frac{1}{T}\int_0^{T} \left|\Delta h_t^{\theta',\theta}\right|^2 dt\, \bone_A\right)^{1/2} \cdot \lim_{T\to\infty}\left(\EE^\theta \frac{1}{T}\int_0^{T} \left|f\left(Y^\lambda_t\right)\right|^2 dt \bone_{A}\right)^{1/2} = 0.
\end{align*}

In addition, using the notation
\begin{equation*}
A_t := \{M_t\geq \epsilon\},
\end{equation*}
and recalling \eqref{eq.MLEConsistPf.Ind.Conv1}, as well as the dominated convergence theorem, we obtain
\begin{align*}
&\lim_{T\to\infty} \EE^\theta\frac{1}{T}\int_0^{T} \left|(\bone_{A_t}-\bone_{A})\,f\left(Y^\lambda_t\right)\cdot\Delta h_t^{\theta',\theta}\right| dt=0,
\end{align*}
which yields
\begin{equation}\label{eq.MLEConsistPf.AveDeltaMu.vanish}
\lim_{T\to\infty} \frac{1}{T}\int_0^{T} \EE^\theta \left(\bone_{A_t}\,f\left(Y^\lambda_t\right) \cdot\Delta h_t^{\theta',\theta}\right) dt = 0.
\end{equation}
Let us expand the integrand in \eqref{eq.MLEConsistPf.AveDeltaMu.vanish}:
\begin{align}
\EE^\theta \left(\bone_{A_t}\,f\left(Y^\lambda_t\right) \cdot\Delta h_t^{\theta',\theta}\right) &= \EE^{\theta}\left(\bone_{A_t}\,f\left(Y^\lambda_t\right) \cdot\tilde{\EE}^{\theta'}\left[h(\theta';X_t)\,\vert\,\mathcal{F}^Y_t\right]\right)
- {\EE}^{\theta}\left(\bone_{A_t}\,f\left(Y^\lambda_t\right)\cdot\tilde{\EE}^{\theta}\left[h(\theta;X_t)\,\vert\,\mathcal{F}^Y_t\right]\right) \nonumber\\
& = \EE^{\theta'}\left[M^{-1}_t\tilde{\EE}^{\theta'}\left(\bone_{A_t}\, f\left(Y^\lambda_t\right)\cdot h(\theta';X_t)\,\vert\,\mathcal{F}^Y_t\right)\right] 
- {\EE}^{\theta}\tilde{\EE}^{\theta}\left(\bone_{A_t}\,f\left(Y^\lambda_t\right)\cdot h(\theta;X_t)\,\vert\,\mathcal{F}^Y_t\right)\nonumber\\
& = \tilde{\EE}^{\theta'}\left(M^{-1}_t \bone_{A_t}\,f\left(Y^\lambda_t\right) \cdot h(\theta';X_t)\right) 
- \tilde{\EE}^{\theta}\left(\bone_{A_t}\,f\left(Y^\lambda_t\right) \cdot h(\theta;X_t)\right),\nonumber
\end{align}
where we used the fact that $M_t = d\PP^{\theta'}_{[0,t]}/d\PP^\theta_{[0,t]}$.
Note that $\tilde\EE^\theta$ is an expectation under $\tilde\PP^\theta$, which is defined on a larger probability space than $\PP^\theta$, hence the right hand side of the above may not equal zero even if $h(\theta;x)$ does not depend on $\theta$.

\smallskip

From the above, we deduce
\begin{align}
&0=\lim_{T\rightarrow\infty}\frac{1}{T}\int_0^{T} \EE^\theta \left(\bone_{A_t}\,f\left(Y^\lambda_t\right) \cdot\Delta h_t^{\theta',\theta}\right) dt \label{eq.MLEConsistPf.AveDeltaMu.vanish.n}\\
&= \lim_{T\rightarrow\infty} \left[\tilde{\EE}^{\theta'} \frac{1}{T}\int_0^{T} M^{-1}_t \bone_{A_t}\, f\left(Y^\lambda_t\right) \cdot h(\theta';X_t) dt
- \tilde{\EE}^{\theta} \frac{1}{T}\int_0^{T} \bone_{A_t}\, f\left(Y^\lambda_t\right) \cdot h(\theta;X_t) dt\right].\nonumber
\end{align}
Recalling that $(h(\theta;X),Y^\lambda)$ is strongly ergodic under $\tilde{\PP}^\theta$ and \eqref{eq.MLEConsistPf.Ind.Conv1}, we obtain $\tilde{\PP}^\theta$-a.s.:
\begin{align*}
&\lim_{T\rightarrow\infty}\frac{1}{T}\int_0^{T} \bone_{A_t}\, f\left(Y^\lambda_t\right)\cdot h(\theta;X_t) dt
= \bone_{\{\xi\geq\epsilon\}} \int_{\RR^m\times \RR^k} z\cdot f(y) \nu^{\theta,\lambda}(dx,dy).
\end{align*}
Using the dominated convergence, we deduce from the above that
\begin{align*}
&\lim_{T\rightarrow\infty}\tilde\EE^\theta \frac{1}{T}\int_0^{T} \bone_{A_t}\, f\left(Y^\lambda_t\right)\cdot h(\theta;X_t) dt
= \PP(\xi\geq\epsilon) \int_{\RR^m\times \RR^k} z\cdot f(y) \nu^{\theta,\lambda}(dz,dy),
\end{align*}
where we also noticed that $\PP^\theta$ and $\tilde\PP^\theta$ coincide on $\mathcal{F}^Y_\infty$ and $\xi\in \mathcal{F}^Y_\infty$.
Similarly, recalling \eqref{eq.MLEConsistPf.Ind.Conv2} and using the strong ergodicity along with the dominated convergence, we obtain
\begin{align*}
&\lim_{T\rightarrow\infty} \tilde{\EE}^{\theta'} \frac{1}{T}\int_0^{T} M^{-1}_t \bone_{A_t}\, f\left(Y^\lambda_t\right)\cdot h(\theta';X_t) dt
= \EE^{\theta'}\left(\eta\bone_{\{\eta\leq\epsilon^{-1}\}}\right) \int_{\RR^m\times \RR^k} z\cdot f(y) \nu^{\theta',\lambda}(dz,dy)\\
&= \lim_{t\rightarrow\infty}\EE^{\theta'}\left(M^{-1}_t\bone_{\{M_t\geq\epsilon\}}\right) \int_{\RR^m\times \RR^k} z\cdot f(y) \nu^{\theta',\lambda}(dz,dy)
= \PP^\theta(\xi\geq\epsilon) \int_{\RR^m\times \RR^k} z\cdot f(y) \nu^{\theta',\lambda}(dz,dy)\,
\end{align*}
where we used again the fact that $M_t = d\PP^{\theta'}_{[0,t]}/d\PP^\theta_{[0,t]}$.

\smallskip

Collecting the above and recalling \eqref{eq.MLEConsistPf.AveDeltaMu.vanish.n}, as well as $\PP^\theta(\xi\geq\epsilon)>0$, we obtain
\begin{align*}
\int_{\RR^m\times \RR^k} z\cdot f(y) \nu^{\theta,\lambda}(dz,dy)
= \int_{\RR^m\times \RR^k} z\cdot f(y) \nu^{\theta',\lambda}(dz,dy).
\end{align*}
The above, in view of the arbitrary choice of $f$ and $\lambda$, yields $\theta=\theta'$ which leads to a desired contradiction.
\qed
\end{proof}

\medskip

To prove Theorem \ref{thm:main} it remains to notice that, due to Proposition \ref{prop:main} and the finiteness of $\Theta$, $\PP^\theta$-a.s., for all large enough $T$, the unique maximizer of $L_T(\cdot)$ is given by $\theta$.

\section{Latent price model of market microstructure}
\label{se:latentprice}

We denote by $X$ the latent price process (a.k.a. the unobserved fundamental price process) and by $Y$ the order flow process, whose value at time $t$ is given by the total volume of all the limit buy orders at time $t$ less the total volume of all limit sell orders.\footnote{This notion of order flow is sometimes called the total order imbalance. Note that it implicitly incorporates the changes in both limit and market orders (as any market order eliminates the corresponding number of limit orders).}
Instead of postulating what $X_t$ should be, as a function of observed market factors, we list the natural properties required of $X$ and then build a simple model that reflects these properties. The three properties we focus on herein are: (i) on a macroscopic scale, $X$ and $Y$ should evolve as random walks, (ii) there is a positive correlation between the local increments of $X$ and $Y$ (referred to as price impact), and (iii) $X$ has a microscopic drift that pushes it away from the midprice (i.e., from the middle point between the best bid and ask prices). The first two properties are standard in the existing literature and are confirmed empirically. 
The third property appeared recently in \cite{N}, where the presence of such microscopic drift is used as an explanation of the concavity of expected price impact of a large meta-order. The presence of such microscopic drift is also confirmed empirically in \cite{N} for $X$ given by the ``microprice" (an observed quantity that is another candidate for the notion of price in a high frequency regime).
Assuming that the prices are measured in cents (i.e., in $\$0.01$) and that the bid-ask spread stays at one cent (its smallest possible value) most of the time, the above properties translate into the following diffusion model for $(X,Y)$:
\begin{equation}\label{eq.mainDyn.def}
\left\{
\begin{array}{l}
{dX_t = \alpha\beta\mu(X_t) dt + \sigma dB_t,}\\
{dY_t = \frac{1}{\alpha} dX_t + \sqrt{\bar\sigma^2 - \sigma^2/\alpha^2} d\tilde W_t = \beta\mu(X_t) dt + \frac{\sigma}{\alpha} dB_t + \sqrt{\bar\sigma^2 - \sigma^2/\alpha^2} d\tilde W_t,\quad Y_0=0,}
\end{array}
\right.
\end{equation}
where
\begin{align*}
\mu(x) = x \,\mmod\,1-1/2,\quad \alpha,\beta,\sigma,\bar\sigma\in(0,\infty),\quad \sigma<\alpha\bar\sigma,
\end{align*}
$B$ and $\tilde W$ are independent one-dimensional Brownian motions, and $X_0$ is a random variable independent of $(B,\tilde W)$.
Using Girsanov's theorem, it is easy to show that there exists a weak solution to \eqref{eq.mainDyn.def}, that this solution is unique in law, and that $X$ and $\tilde W$ are independent.
The best bid and ask prices at time $t$ are given by $P^b_t:=B(X_t)$ and $P^a_t:=A(X_t)$ respectively, with the associated functions defined as follows:
\begin{align}
&B(x) = 
\left\{
\begin{array}{ll}
{\lfloor x \rfloor,} & {\text{if }\nexists i\in\ZZ\text{ s.t. }|x-i|\leq \epsilon,}\\
{i-1,} & {\text{if }\exists i\in\ZZ\text{ s.t. }|x-i|\leq \epsilon,}
\end{array}
\right.\label{eq.B.def}\\
&A(x) = 
\left\{
\begin{array}{ll}
{\lceil x \rceil,} & {\text{if }\nexists i\in\ZZ\text{ s.t. }|x-i|\leq \epsilon,}\\
{i+1,} & {\text{if }\exists i\in\ZZ\text{ s.t. }|x-i|\leq \epsilon,}
\end{array}
\right.\label{eq.A.def}
\end{align}
for a constant $\epsilon\in(0,1/2)$.

\medskip

For $\beta=0$, the system \eqref{eq.mainDyn.def} yields a classical model (in the spirit of \cite{AlmgrenChriss}) where the latent price and the order flow are modeled as correlated random walks, and where each share purchased moves the latent price up by $\alpha$ (the latter is the price impact coefficient). When $\beta>0$, the latent price process obtains an additional micro-drift (which is also transferred to the order flow via the price impact) that pushes the latent price away from the midprice.

\medskip

Our observations are given by the continuous-time processes
\begin{align*}
P^b_t, P^a_t, Y_t,\quad t\in[0,T].
\end{align*}
Note that the present setting does not immediately fit the one considered in Section \ref{se:general}, because the observation process herein includes certain (non-invertible) functions of $X$, in addition to the coupled diffusion process $Y$. In what follows, we show how to reduce the estimation problem in the above model to the one considered in Section \ref{se:general}.

We start by noticing that the constant $\bar\sigma^2$ can be estimated precisely from the above sample, for any $T>0$, as
\begin{align}
\frac{[Y,Y]_T}{T}=\bar\sigma^2,\label{eq.sigmaBar.est}
\end{align}
we assume that $\bar\sigma^2$ is known.
Our primary goal is to find a consistent estimator for the unknown vector of parameters $\theta:=(\alpha,\beta,\sigma^2)$, where the consistency is understood in the regime $T\rightarrow\infty$.\footnote{It is also shown in Subsection \ref{subse:empirical} how to construct a consistent estimator of $\epsilon$. However, $\epsilon$ does not appear in the estimation algorithm for the rest of the parameters -- this is why the dependence on $\epsilon$ is omitted in most of the remainder of the paper.}
We assume that $\theta$ belongs to a given (known) set $\Theta\subset (0,\infty)^3$. We construct our estimator on the space $(\ZZ^2)^{\otimes[0,\infty)}\times C([0,\infty);\RR)$ of the paths of $(P^b,P^a,Y)$. For any $\theta\in\Theta$, we equip the latter space with the measure $\bar\PP^{\theta}$ which is defined as the distribution of $(B(X),A(X),Y)$, where $(X,Y)$ are given by \eqref{eq.mainDyn.def}. 

\medskip

The estimation procedure consists of two parts. First, we obtain the consistent estimator of 
\begin{align}
&\Sigma=\Sigma(\theta):=\frac{\sigma^2}{\phi(\gamma(\theta))},\label{eq.Sigma.def}
\end{align}
where
\begin{align}
&\phi(z):=2\int_{0}^1\exp(-z (x-1/2)^2)\int_0^x \exp(z (y-1/2)^2) dy dx,\quad \gamma(\theta):=\frac{\alpha\beta}{\sigma^2}.
\label{eq.C1C2.def}
\end{align}

\smallskip


\begin{proposition}\label{prop:Sigma}
Assume that $M:\RR_+\rightarrow\NN$ is such that $M(T),T/M(T)\rightarrow\infty$ as $T\rightarrow\infty$. Then, for any $\theta=(\alpha,\beta,\sigma^2)\in\Theta$, we have
\begin{equation*}
\lim_{T\rightarrow\infty}\hat\Sigma_T = \Sigma(\theta),
\end{equation*}
in probability under $\bar\PP^{\theta}$, where
$$
\hat\Sigma_T:=\frac{1}{T} \sum_{k=1}^{M(T)} \left(\hat X_{kT/M(T)} - \hat X_{(k-1)T/M(T)}\right)^2,
\quad \hat X_t:=\frac{1}{2}(P^b_t+P^a_t).
$$
\end{proposition}

The proof of this proposition is given in the Appendix.




\medskip

The second, and most important, part of the estimation is performed by applying the maximum likelihood method to the partially observed diffusion \eqref{eq.mainDyn.def}, as described in Section \ref{se:general}. Note that the notation used herein is consistent with the one used in Section \ref{se:general}:
\begin{align*}
\bar\sigma W = \frac{\sigma}{\alpha} B + \sqrt{\bar\sigma^2 - \sigma^2/\alpha^2}\,\tilde W,
\quad b(\theta;x) = \alpha\beta\mu(x),\quad h(\theta;x)=\beta\mu(x),\quad \sigma(\theta;x)=\sigma,
\quad q=m=d=1.
\end{align*}
Recall that $\tilde \PP^{\theta}$ denotes the distribution of $(X,Y)$, given by \eqref{eq.mainDyn.def} for any fixed $\theta\in\Theta$, that $\PP^{\theta}$ is the second marginal of $\tilde \PP^{\theta}$, and that $\PP$ is the distribution of $\bar\sigma W$ (i.e., a scaled Wiener measure).
Recall also that $\tilde\PP^\theta_T,\PP^{\theta}_T,\PP_T$ denote the restrictions of respective measures to the time interval $[0,T]$.

The likelihood function, in this case, becomes:
\begin{align}
L_T(\theta) =\exp\left(-\frac{\beta^2}{2\bar\sigma^2}\int_0^T \left(\mu^{\theta}_t\right)^2 dt
+ \frac{\beta}{\bar\sigma^2} \int_0^T \mu^{\theta}_t dY_t \right),
\quad \mu^\theta_t := \tilde\EE^\theta\left( \mu(X_t)\,\vert\,\mathcal{F}^Y_t\right).
\label{eq.LatentPrice.L.def}
\end{align}
Notice that the moment-matching condition $\hat\Sigma_T=\Sigma(\theta)$, effectively, reduces the dimension of unknown $\theta$.
Thus, for any given $T>0$ and $s>0$, we define MLE $\hat{\theta}_T(s)$ of $\theta$ for a reduced parameter set as any random variable satisfying
\begin{align}
&\hat{\theta}_T(s) \in \text{argmax}_{\theta\in\Theta(s)}\left(-\frac{\beta}{2}\int_0^T \left(\mu^{\theta}_t\right)^2 dt
+ \int_0^T \mu^{\theta}_t dY_t \right),\label{eq.alphahat.def}
\end{align}
where 
\begin{align*}
& \Theta(s):=\text{argmin}_{\theta\in\Theta}\left(|\Sigma(\theta)-s|\right).
\end{align*}
More precisely, we define $\hat{\theta}_T(s,\omega)$ as a selector from the above argmax that is $\mathcal{B}(\RR)\otimes\mathcal{F}^Y_T$-measurable in $(s,\omega)$ (see \cite[Thm 18.13]{Aliprantis} for the existence of such a measurable selector).

The above MLE for a reduced parameter set is consistent under the following technical assumption (in addition to the assumption $|\Theta|<\infty$).

\begin{ass}\label{ass:1}
The function
\begin{align}
z\mapsto  \frac{1-1/\phi(z)}{z(1/\psi(z)-1) \phi(z)}
\label{eq.TheFunc}
\end{align}
is invertible on the domain $\{z=\alpha\beta/\sigma^2:\,(\alpha,\beta,\sigma^2)\in\Theta\}$,
where
\begin{align*}
& \phi(z)=2\int_{0}^1\exp(-z (x-1/2)^2)\int_0^x \exp(z (y-1/2)^2) dy dx,
\quad \psi(z):= \int_0^1 \exp(z y^2 - z y) dy.
\end{align*}
\end{ass}

\begin{remark}
Assumption \ref{ass:1}, in principle, needs to be verified for every given set $\Theta$, which is not hard to do if the latter set is finite. The numerical experiments indicate that $\phi(z)$ is strictly increasing for all $z\geq0$ (see Figure \ref{fig:1}). Such strict monotonicity of \eqref{eq.TheFunc} would clearly imply that Assumption \ref{ass:1} is always satisfied. However, at this time we are not able to prove this monotonicity rigorously for the entire range $z\geq0$, hence it remains a conjecture.
\end{remark}

\begin{theorem}\label{thm:main}
Assume that $\Theta$ is finite and let Assumption \ref{ass:1} hold.
Then, for any $\theta=(\alpha,\beta,\sigma^2)\in\Theta$ and $s=\Sigma(\theta)$, we have
\begin{equation*}
\lim_{T\rightarrow\infty} \hat{\theta}_T(s) = \theta,
\end{equation*}
in probability (and a.s.) under $\PP^{\theta}$.
\end{theorem}

The above theorem says that, if we know the true value of $\Sigma(\theta)$, denoted by $s$ (without necessarily knowing the true value of $\theta$ itself), then a MLE for the reduced parameter set $\Theta(s)=\{\theta\in\Theta:\,\Sigma(\theta)=s\}$ is consistent. The proof of Theorem \ref{thm:main} consists of verification of the assumptions of Theorem \ref{thm:NY}, and it is presented in the next subsection.
\medskip

Finally, we combine the above results to obtain a consistent estimator of $\theta$.
To state this final result, we note that $\hat{\theta}_T(\hat\Sigma_T)$ is a random variable on the probability space $(\ZZ^2)^{\otimes[0,\infty)}\times C([0,\infty),\RR)$ of the paths of $(P^b,P^a,Y)$.\footnote{To avoid confusion, we note that $\hat{\theta}_T(s)$ is constructed as an estimator in a model where the observation is artificially restricted to the path of $Y$ only: i.e., $\hat{\theta}_T(s)$ is a random variable on the canonical space consisting of the paths of $Y$. On the other hand, $\hat{\theta}_T(\hat\Sigma_T)$ is an estimator in a model with actual observations, given by $(P^b,P^a,Y)$: i.e., $\hat{\theta}_T(\hat\Sigma_T)$ is a random variable on the canonical space consisting of the paths of the latter processes.}

\begin{cor}
Assume that $\Theta$ is finite and let Assumption \ref{ass:1} hold.
Then, for any $\theta=(\alpha,\beta,\sigma^2)\in\Theta$, we have
\begin{equation*}
\lim_{T\rightarrow\infty} \hat{\theta}_T\left(\hat\Sigma_T\right) = \theta,
\end{equation*}
in probability (and a.s.) under $\bar\PP^{\theta}$.
\end{cor}
\begin{proof}
We only need to note that, since $\Theta$ is finite and $\hat\Sigma_T\rightarrow\Sigma(\theta)$, we have $\bar\PP^{\theta}$-a.s.: for large enough $T$,
\begin{align*}
&  \Theta(\hat\Sigma_T)
= \{\theta'\in\Theta:\, \Sigma(\theta')=\Sigma(\theta)\}.
\end{align*}
Then, for such $T$, we have $\hat{\theta}_T(\hat\Sigma_T)=\hat{\theta}_T(\Sigma(\theta))$, and the result follows from Theorem \ref{thm:main}.
\qed
\end{proof}

\subsection{Proof of Theorem \ref{thm:main}}
\label{se:pf.Thm.main}

Our goal is to deduce the statement of Theorem \ref{thm:main} from Theorem \ref{thm:NY}. First, we define
\begin{align*}
dY^{\lambda}_t = -\lambda Y^{\lambda}_t dt + dY_t,\quad Y^{\lambda}_0=0=Y_0,
\end{align*}
for all $\lambda\in(0,\infty)=:\Lambda$.
It is clear that $Y^\lambda$ is adapted to $\mathbb{F}^Y$.
The following proposition shows that $(X\,\mmod\,1,Y^\lambda)$ is strongly ergodic, which verifies the first assumption of Theorem \ref{thm:NY}. In this proposition, we recall the notation $H^1([a,b],\RR)$ which denotes the Sobolev space of square-integrable functions from $[a,b]$ to $\RR$ whose first weak derivative is also square-integrable. The latter space is equipped with its natural Sobolev norm.

\begin{proposition}\label{le:ergod}
For every $\theta=(\alpha,\beta,\sigma^2)\in\Theta$ and $\lambda>0$ there exists a probability measure $\nu^{\theta,\lambda}$ on $\RR^{2}$ s.t.
\begin{align}
\int_{\RR^{2}}\phi(z,y) \nu^{\theta,\lambda}(dz,dy) 
= \lim_{T\rightarrow\infty}  \frac{1}{T} \int_0^T \tilde\EE^\theta \phi(\beta\mu(X_t), Y^\lambda_t) dt
= \lim_{T\rightarrow\infty} \frac{1}{T} \int_0^T \phi(\beta\mu(X_t),Y^\lambda_t) dt,\quad \tilde\PP^\theta\text{-a.s.},
\label{eq.LatentPrice.MainThm.Pf.Le1.eq1}
\end{align} 
for any bounded Borel $\phi:\RR^{2}\rightarrow\RR$.

Moreover, the density $q(t,\cdot)$ of $X_t\,\mmod\,1$ under $\tilde\PP^\theta$ satisfies
\begin{align*}
\lim_{t\rightarrow\infty} \int_t^{t+u} \|q(s,\cdot)-\chi\|_{H^1([0,1],\RR)} ds = \lim_{t\rightarrow\infty} \|q(t,\cdot)-\chi\|_{L^2([0,1],\RR)} = 0,
\end{align*}
where
\begin{align}
\chi(x):=\frac{\exp(\gamma x^2 - \gamma x)}{\psi(\gamma)}\bone_{[0,1]}(x),
\quad \psi(\gamma)= \int_0^1 \exp(\gamma y^2 - \gamma y) dy,
\quad \gamma=\gamma(\theta)=\alpha\beta/\sigma^2,
\label{eq.LatentPrice.mainThm.Pf.chi.def}
\end{align}
and the above convergence is uniform over all distributions of $X_0\,\mmod\,1$ and locally uniform over $u\geq0$.
\end{proposition}
\begin{proof}
{\bf Step 1}.
In this step, we prove the first equality in \eqref{eq.LatentPrice.MainThm.Pf.Le1.eq1}. To this end, we apply It\^o's formula to $(Y^\lambda)^2$, 
\begin{align*}
d \tilde\EE^\theta (Y^\lambda_t)^2 = \left[ 2 \tilde \EE^\theta \mu(X_t) Y^\lambda_t - \lambda \tilde\EE^\theta (Y^\lambda_t)^2 + \bar\sigma^2\right] dt,
\end{align*}
and use Young's inequality to conclude that $\tilde\EE^\theta (Y^\lambda_t)^2$ is bounded from above uniformly over all $t\geq1$. This yields tightness of the laws of $(\beta\mu(X_t), Y^\lambda_t)$ over $t\geq0$, which in turn yields tightness of $\{\nu_t\}_{t\geq0}$, where
\begin{align*}
\nu_T(A) := \frac{1}{T} \int_0^T \tilde\PP^\theta\left((\beta\mu(X_t), Y^\lambda_t)\in A\right) dt,
\end{align*}
for any Borel $A\subset \RR^2$. Thus, we conclude that any sequence $T_n\rightarrow\infty$ contains a subsequence along which $\nu_T$ converges weakly to a probability measure on $\RR^2$. Any such limiting measure is invariant for the Markov process $(\beta\mu(X), Y^\lambda)$, as is easy to see from the definition of $\nu_T$.

\smallskip

Let us show that there exists only one invariant measure of $(\beta\mu(X), Y^\lambda)$. We argue by contradiction and assume that there exist two distinct invariant measures $\nu$ and $\tilde \nu$. 
Notice that
\begin{align*}
Y^\lambda_t = \frac{1}{\alpha} \int_0^t e^{-\lambda(t-s)} dX_s + \sqrt{\bar\sigma^2 - \sigma^2/\alpha^2} \int_0^t e^{-\lambda(t-s)} d\tilde W_s,
\end{align*}
where the Brownian motion $\tilde W$ is independent of $X$. Therefore, fixing any deterministic initial condition for $(X,Y^\lambda)$ and any $t,N>0$, we have
\begin{align*}
&\tilde\EE^\theta \left[ \bone_{[a,b]}(Y^\lambda_t)\,\vert\,X_t\right]
= \tilde\EE^\theta \left[ \tilde\EE^\theta \left[ \bone_{[a,b]}(Y^\lambda_t) \,\vert\, \mathcal{F}^X_t\right]\,\vert\,X_t\right],\\
& (b-a) C_1 \leq \tilde\EE^\theta \left[ \bone_{[a,b]}(Y^\lambda_t) \,\vert\, \mathcal{F}^X_t\right]\leq (b-a) C_2,
\end{align*}
for some constants $C_i(N)>0$ and for all $-N\leq a<b\leq N$.
Thus, the conditional distribution of $Y^\lambda_t$ given $X_t$ admits a strictly positive density. As the distribution of $X_t$ has the same properties (see Step 3), we conclude that the transition kernel of the Markov process $(\beta\mu(X),Y^\lambda)$ has a strictly positive density denoted by $K$.
Then, $\nu$ and $\tilde \nu$ must have densities $p$ and $\tilde p$, and
\begin{align*}
\int_{\RR^2} K(z,y,t,z',y') g(z,y) dz dy = g(z',y'),\quad g:=p-\tilde p.
\end{align*}
Notice that $\int_{\RR^2} g(z,y) dz dy=0$ and $g$ cannot be equal to zero a.e. by our assumption. Hence, there exist two sets of strictly positive Lebesgue measures on which $g<0$ and $g>0$. This and the strict positivity of $K$ imply that
\begin{align*}
&\int_{\RR^2} |g(z',y')| dz' dy'
= \int_{\RR^2} \left|\int_{\RR^2} K(z,y,t,z',y') g(z,y) dz dy\right| dz' dy'\\
&> \int_{\RR^2} \int_{\RR^2} K(z,y,t,z',y') |g(z,y)| dz dy dz' dy'
= \int_{\RR^2} |g(z,y)| dz dy,
\end{align*}
which yields the desired contradiction.
Thus, the invariant measure of $(\beta\mu(X),Y^\lambda)$ is unique and, therefore, $\nu_T$ converges weakly as $T\rightarrow\infty$, which yields the first equality in \eqref{eq.LatentPrice.MainThm.Pf.Le1.eq1}. Note that this equality holds for any choice of the initial condition for $(X,Y^\lambda)$.

\medskip

{\bf Step 2}.
In this step, we prove the second equality in \eqref{eq.LatentPrice.MainThm.Pf.Le1.eq1} by showing that the variance of the time-average in its right hand side vanishes as $T\rightarrow\infty$.
Let us fix $\theta$ and a bounded Lipschitz $\phi$, and denote
\begin{align*}
\phi_t:=\tilde\EE^\theta \phi(\beta\mu(X_t),Y^\lambda_t).
\end{align*}
Then,
\begin{align*}
& \tilde\EE^\theta\left[\frac{1}{T} \int_0^T \phi(\beta\mu(X_t),Y^\lambda_t) dt - \tilde\EE^\theta\frac{1}{T} \int_0^T \phi(\beta\mu(X_t),Y^\lambda_t) dt\right]^2\\
&= \frac{2}{T^2} \int_0^T\int_0^t \tilde\EE^\theta \left(\phi(\beta\mu(X_t),Y^\lambda_t)-\phi_t\right) \left(\phi(\beta\mu(X_s),Y^\lambda_s)-\phi_s\right) ds dt\\
&= \frac{2}{T^2} \int_0^T\int_0^t \tilde\EE^\theta\left[\tilde\EE^\theta \left[ \phi(\beta\mu(X_t),Y^\lambda_t)-\phi_t\,\vert\, X_s,Y^\lambda_s\right] \left(\phi(\beta\mu(X_s),Y^\lambda_s)-\phi_s\right)\right] ds dt\\
&= \tilde\EE^\theta \frac{2}{T} \int_0^T \frac{1}{T}\int_0^{T-s}\tilde\EE^\theta \left[ \phi(\beta\mu(X_{s+u}),Y^\lambda_{s+u})-\phi_{s+u}\,\vert\, X_s,Y^\lambda_s\right] du \left(\phi(\beta\mu(X_s),Y^\lambda_s)-\phi_s\right)ds.
\end{align*}

Let us analyze the conditional expectation in the right hand side of the above, using the Markov property of $(X,Y^\lambda)$:
\begin{align*}
&\tilde\EE^\theta \left[ \phi(\beta\mu(X_{s+u}), Y^\lambda_{s+u})-\phi_{s+u}\,\vert\, X_s, Y^\lambda_s\right]
= g(u,X_s, Y^\lambda_s) - \tilde\EE^\theta g(u,X_s, Y^\lambda_s),\\
& g(u,x,y) := \tilde\EE^\theta \left[ \phi(\beta\mu(X_{u}),Y^\lambda_{u})\,\vert\, X_0=x,Y^\lambda_0=y\right].
\end{align*}
Notice that $g(u,\cdot,y)$ is 1-periodic and bounded uniformly over all $(u,y)$, and that 
\begin{align*}
\left|g(u,x,y) - g(u,x,y')\right| \leq C_1 |y-y'| e^{-\lambda u},\quad \forall\,(u,x,y).
\end{align*}
In addition, using the strong Markov property of $(X,Y^\lambda)$, we obtain, for any $x\in[0,1]$:
\begin{align*}
g(u,x,y) = \tilde\EE^\theta g\left((u-\tau)^+,X^x_{\tau\wedge u}, Y^{\lambda,y}_{\tau\wedge u}\right),\quad \tau=\tau(x):=\inf\{t\geq0:\, X^x_t\in\{0,1\}\},
\end{align*}
where $X^x$ and $Y^{\lambda,y}$ are started from $x$ and $y$ respectively.
Then,
\begin{align*}
&g(u,x,y) - g(u,0,0) = \tilde\EE^\theta \bone_{\{\tau(x)\leq u/2\}} \left[g\left(u-\tau(x),0,Y^{\lambda,y}_{\tau(x)\wedge u}\right) - g(u,0,0)\right]
+ O\left(\tilde\PP^\theta(\tau(x)>u/2)\right)\\
&= \tilde\EE^\theta \bone_{\{\tau(x)\leq u/2\}} \left[g\left(u-\tau(x),0,0\right) - g(u,0,0)\right]
+ O\left( e^{-\lambda u/2} \tilde\EE^\theta |Y^{\lambda,y}_{\tau(x)\wedge u}|
+ \tilde\PP^\theta(\tau(x)>u/2)\right)\\
&= \tilde\EE^\theta \bone_{\{\tau(x)\leq u/2\}} \left[g\left(u-\tau(x),0,0\right) - g(u,0,0)\right]
+ O\left( e^{-\lambda u/2} (1+|y|)
+ \tilde\PP^\theta(\tau(x)>u/2)\right),\\
& \left|\frac{1}{T}\int_0^T \tilde\EE^\theta \bone_{\{\tau\leq u/2\}} \left[g\left(u-\tau,0,0\right) - g(u,0,0)\right] du\right|\\
&=  \tilde\EE^\theta \left[\frac{1}{T}\int_{\tau}^{T-\tau} g\left(u,0,0\right) du - \frac{1}{T}\int_{2\tau}^T g(u,0,0)du \right]
\leq \frac{C_2}{T} \sup_{z\in[0,1]} \tilde\EE^\theta \tau(z).
\end{align*}
The above yields
\begin{align*}
& \frac{1}{T} \int_0^T g(u,x,y) du = \frac{1}{T} \int_0^T g(u,0,0) du + O\left((1+|y|)/T\right),\quad \forall\,x,y,T.
\end{align*}
Then, as shown in Step 1,
\begin{align*}
& \int_{\RR^{2}}\phi(z,y) \nu^{\theta,\lambda}(dz,dy) 
= \lim_{T\rightarrow\infty}  \frac{1}{T} \int_0^T g(u,0,0) du,\\
& \frac{1}{T} \int_0^T \tilde\EE^\theta g(u,X_s,Y^\lambda_s) du
= \frac{1}{T} \int_0^T g(u,0,0) du + O\left(1/T\right)
= \frac{1}{T} \int_0^T g(u,X_s,Y^\lambda_s) du + O\left((1+|Y^\lambda_s|)/T\right).
\end{align*}

Collecting the above, we obtain, for any $\varepsilon>0$:
\begin{align*}
& \tilde\EE^\theta\left[\frac{1}{T} \int_0^T \phi(\beta\mu(X_t),Y^\lambda_t) dt - \tilde\EE^\theta\frac{1}{T} \int_0^T \phi(\beta\mu(X_t),Y^\lambda_t) dt\right]^2\\
&= \tilde\EE^\theta \frac{2}{T} \int_0^T \frac{1}{T}\int_0^{T-s} \left(g(u,X_s,Y^\lambda_s) - \tilde\EE^\theta g(u,X_s,Y^\lambda_s)\right) du \left(\phi(\beta\mu(X_s),Y^\lambda_s)-\phi_s\right)ds\\
&\leq C_3 \varepsilon
+ C_4 \tilde\EE^\theta \frac{2}{T} \int_0^{T(1-\varepsilon)} \frac{T-s}{T} \left| \frac{1}{T-s}\int_0^{T-s} \left(g(u,X_s,Y^\lambda_s) - \tilde\EE^\theta g(u,X_s,Y^\lambda_s)\right) du\right| ds\\
&\leq C_3 \varepsilon
+ C_5 \tilde\EE^\theta \frac{2}{\varepsilon T^2} \int_0^{T(1-\varepsilon)} \tilde\EE^\theta (1+|Y^\lambda_s|) ds
\rightarrow C_3 \varepsilon,
\end{align*}
as $T\rightarrow\infty$. Since $\varepsilon>0$ can be chosen arbitrarily small, we conclude that the second equality in \eqref{eq.LatentPrice.MainThm.Pf.Le1.eq1} holds.
As a side remark, we note that the variance of the time-average in the right hand of the second equality in \eqref{eq.LatentPrice.MainThm.Pf.Le1.eq1} vanishes uniformly over all distributions of $X_0$.

\medskip

{\bf Step 3}.
In this step, we prove the second statement of the proposition. Recall the scale function $S$ that transforms $X$ into a driftless diffusion:
\begin{align*}
&Z_t:=S(X_t),\quad S(x):=\int_0^x \exp\left(-2\int_{1/2}^y \mu(z)/\sigma^2 dz\right) dy
=\int_0^x \exp\left(-\mu^2(y)/\sigma^2\right) dy,\\
& dZ_t=\exp\left(-2\int_{1/2}^{X_t} \mu(z)/\sigma^2(z) dz\right) dX_t 
- \exp\left(-2\int_{1/2}^{X_t} \mu(z)/\sigma^2(z) dz\right) \mu(X_t) dt\\
& = \tilde\sigma(Z_t) dW_t,
\quad \tilde\sigma(z):= S'(S^{-1}(z)) \sigma.
\end{align*}
Note that $\tilde\sigma$ is p-periodic, with $p:=S(1)>0$.
It is easy to see that the second statement of the theorem is equivalent to the following statement: as $t\rightarrow\infty$, the density of $Z_t\,\mmod\,p$ under $\tilde\PP^\theta$ converges in $H^1([0,p],\RR)$ to $\tilde\sigma^{-2}$, uniformly over all initial conditions in $[0,p]$.

As $\tilde\sigma$ is bounded and Lipschitz, the results of \cite{Friedman} yield, for any $T>0$, the existence of the fundamental solution $\Gamma(x,t,y)$ of the PDE
\begin{align*}
\partial_t \Gamma = \frac{1}{2} \tilde\sigma^2(x) \partial^2_{xx} \Gamma,\quad x\in\RR,\,\, t>0.
\end{align*}
Then, it is easy to verify that
\begin{align*}
\hat q(x,t,y):= \frac{\tilde\sigma^2(x)}{\tilde\sigma^2(y)} \Gamma(y,t,x)
\end{align*}
is the fundamental solution to the adjoint equation in the $(t,y)$ variables:
\begin{align}
\partial_t \hat q = \frac{1}{2} \partial^2_{yy} \left[\tilde\sigma^2(y)\hat q\right],\quad y\in\RR,\,\, t>0,
\label{eq.tildeq.adjoint}
\end{align}
such that $\tilde\sigma^2(\cdot)\hat q(x,\cdot,\cdot)\in C^{1,2}((0,\infty)\times\RR)$ and that the weak derivative $\partial_y\hat q(x,\cdot,\cdot)$ is represented by a locally integrable function on $(0,\infty)\times\RR$.

The defining properties of a fundamental solution imply that $\hat q$ is the transition density of the Markov process $Z$.
In addition, $\hat q$ is strictly positive and the Gaussian estimates for $\Gamma$, established in \cite{Friedman} imply:
\begin{align}
\hat q(x,t,y) \leq C_1 t^{-1/2} \exp(-C_2(x-y)^2/t),\quad \left|\partial_y\hat q(x,t,y)\right| \leq C_1 t^{-3/2}|x-y| \exp(-C_2(x-y)^2/t).
\label{eq.Gaussian.est}
\end{align}
As $X_t:=S^{-1}(Z_t)$ and the derivative of $S^{-1}$ is Lipschitz and bounded from above and away from zero, we note (for the future reference) that the Gaussian estimates \eqref{eq.Gaussian.est} also hold for the transition density $\bar q(x,t,y)$ of the Markov process $X$ in place of $\hat q$.


Denoting by $\kappa(dx)$ the distribution of $Z_0$, we conclude that the density $\tilde q(t,\cdot)$ of $Z_t\,\mmod\,p$ is given by
\begin{align*}
\tilde q(t,y) = \sum_{n\in\ZZ} \int_\RR \hat q(x,t,y+np) \kappa(dx),\quad y\in[0,p].
\end{align*}
Using the Gaussian estimates \eqref{eq.Gaussian.est}, we deduce that $\tilde q$ is well defined. Next, we notice that, for any constant $C>0$, the function $C\tilde\sigma^{-2}$ solves \eqref{eq.tildeq.adjoint}.
Then, using the regularity of $\hat q(x,\cdot,\cdot)$ and the $p$-periodicity of $\tilde\sigma$, we conclude that the function $g(t,y):=\tilde q(t,y)-C\tilde\sigma^{-2}(y)$ satisfies the following: $g\in L^2_{loc}((0,\infty),H^{1}([0,p],\RR))$, that $g(t,0)=g(t,p)$, $\partial_y\left[\tilde\sigma^2 g\right](t,0^+)=\partial_y\left[\tilde\sigma^2 g\right](t,p^-)$, and that 
\begin{align}
& \int_0^p g^2(t,y) \tilde\sigma^2(y) dy - \int_0^p g^2(s,y) \tilde\sigma^2(y) dy = -\int_s^t \int_0^p \left(\partial_y\left[\tilde\sigma^2(y) g(r,y) \right]\right)^2 dy dr,\quad 1\leq s\leq t.
\label{eq.Le1.EnergyEst}
\end{align}
It is clear that the choice $C=(\int_0^p \tilde\sigma^{-2}(y) dy)^{-1}$ ensures $\int_0^p g(t,y) dy=0$.

It suffices to show that $g(t,\cdot)$ converges to zero as $t\rightarrow\infty$. To this end, using the notation $\tilde g(r,y):=\tilde\sigma^2(y) g(r,y)$, we notice that
\begin{align*}
& \tilde g^2(t,y) = \left[\left( \int_0^p \tilde\sigma^{-2}(z) dz\right)^{-1} \int_0^p \tilde g(r,y) \tilde\sigma^{-2}(z) dz\right]^2 
=\left[\left( \int_0^p \tilde\sigma^{-2}(z) dz\right)^{-1} \int_0^p (\tilde g(r,y)-\tilde g(r,z)) \tilde\sigma^{-2}(z) dz\right]^2\\ 
&\leq \left( \int_0^p \tilde\sigma^{-2}(z) dz\right)^{-1} \int_0^p (\tilde g(r,y)-\tilde g(r,z))^2 \tilde\sigma^{-2}(z) dz
= \left( \int_0^p \tilde\sigma^{-2}(z) dz\right)^{-1} \int_0^p \left(\int_z^y \partial_x \tilde g(r,x) dx\right)^2 \tilde\sigma^{-2}(z) dz\\
&\leq \left( \int_0^p \tilde\sigma^{-2}(z) dz\right)^{-1} \int_0^p |y-z| \int_z^y \left(\partial_x \tilde g(r,x)\right)^2 dx \tilde\sigma^{-2}(z) dz
\leq p \int_0^p \left(\partial_x \tilde g(r,x)\right)^2 dx,\quad y\in[0,p].
\end{align*}
Hence, there exists a constant $C_3>0$ s.t.
\begin{align*}
&C_3 \int_0^p g^2(r,y) \tilde\sigma^2(y) dy \leq \int_0^p \left(\partial_y\left[\tilde\sigma^2(y) g(r,y) \right]\right)^2 dy,
\end{align*}
which yields
\begin{align*}
\frac{d}{dt}\int_0^p g^2(t,y) \tilde\sigma^2(y) dy
\leq - C_3 \int_0^p g^2(t,y) \tilde\sigma^2(y) dy,
\end{align*}
which in turn implies, via Gronwall's inequality, that
\begin{align*}
\int_0^p g^2(t,y) \tilde\sigma^2(y) dy
\leq \int_0^p g^2(1,y) \tilde\sigma^2(y) dy\, e^{- C_3 (t-1)},\quad t\geq 1.
\end{align*}
Using the Gaussian estimates \eqref{eq.Gaussian.est} and the definition of $\tilde q$, it is easy to see that the $L^2$ norm of $g(1,\cdot)$ can be bounded from above uniformly over all $\kappa$.
This implies that $g(t,\cdot)$ converges to zero in $L^2$ as $t\rightarrow\infty$.
Combining this conclusion with \eqref{eq.Le1.EnergyEst}, we obtain the second statement of the proposition.
\qed
\end{proof}

\medskip

Next, we verify the second assumption of Theorem \ref{thm:NY}.

\begin{proposition}
\label{prop:LatentPrice.verify}
Assume that $\Theta$ is finite, let Assumption \ref{ass:1} hold, and let $\nu^{\theta,\lambda}$ be as in Proposition \ref{le:ergod}.
Then, for any $\theta,\theta'\in\Theta$ with $\Sigma(\theta)=\Sigma(\theta')$ the equality
\begin{align}\label{eq.Background.nu.sep.inLe}
\int_{\RR^{2}} z f(y)\, \nu^{\theta,\lambda}(dz,dy)=\int_{\RR^{2}} z f(y)\, \nu^{\theta',\lambda}(dz,dy)
\quad \forall\,\lambda>0 \text{ and bounded Borel }f:\RR\rightarrow\RR
\end{align}
implies $\theta=\theta'$.
\end{proposition}
\begin{proof}
The uniform bound on $\tilde\EE^\theta (Y^\lambda_t)^2$ established in the proof of Proposition \ref{le:ergod} implies that the first equality in \eqref{eq.LatentPrice.MainThm.Pf.Le1.eq1} can be extended to all polynomially bounded $\phi$. 
Thus, we obtain, for any $\lambda>0$,
\begin{align*}
&\beta\lim_{T\rightarrow\infty} \frac{1}{T} \int_0^T \tilde\EE^\theta \mu(X_t) Y^\lambda_t dt
= \int_{[-\beta/2,\beta/2]\times \RR} z y \nu^{\theta,\lambda}(dx,dy)\\
&= \int_{[-\beta'/2,\beta'/2]\times \RR} z y \nu^{\theta',\lambda}(dx,dy)
= \beta'\lim_{T\rightarrow\infty} \frac{1}{T} \int_0^T \tilde\EE^{\theta'} \mu(X_t) Y^\lambda_t dt.
\end{align*}
Next, we notice
\begin{align*}
& Y^\lambda_t = \int_0^t e^{-\lambda(t-s)} dY_s 
= \frac{1}{\alpha}X^\lambda_t + \sqrt{\bar\sigma^2 - \sigma^2/\alpha^2}\int_0^t e^{-\lambda(t-s)} d\tilde W_t,
\end{align*}
where $\tilde W$ is a Brownian motion independent of $X$ and
\begin{align*}
& X^\lambda_t := \int_0^t e^{-\lambda(t-s)} dX_s = \alpha\beta \int_0^t e^{-\lambda(t-s)} \mu(X_s) ds + \sigma \int_0^t e^{-\lambda(t-s)} dW_s.
\end{align*}
Then, 
\begin{align}
&\frac{\beta}{\alpha}\lim_{T\rightarrow\infty} \frac{1}{T} \int_0^T \tilde\EE^\theta \mu(X_t) X^\lambda_t dt
=\beta \lim_{T\rightarrow\infty} \frac{1}{T} \int_0^T \tilde\EE^\theta \mu(X_t) Y^\lambda_t dt\nonumber\\
&=\beta' \lim_{T\rightarrow\infty} \frac{1}{T} \int_0^T \tilde\EE^{\theta'} \mu(X_t) Y^\lambda_t dt
= \frac{\beta'}{\alpha'} \lim_{T\rightarrow\infty} \frac{1}{T} \int_0^T \tilde\EE^{\theta'} \mu(X_t) X^\lambda_t dt,
\label{eq.LatentPrice.MainThm.Pf.Le2.eq1}
\end{align}
for all $\lambda>0$.

Our goal is to deduce tractable equations for $\theta$ and $\theta'$ based on the above equality. We achieve this by analyzing the above in two different asymptotic regimes.

\medskip

{\bf Step 1}. First, we consider the regime $\lambda\rightarrow\infty$.
It is easy to see that $\sup_{t\geq0}\EE (X^\lambda_t)^2<\infty$.
Using the latter observation and applying It\^o's rule to $(X^\lambda)^2$,
\begin{align*}
d \tilde\EE^\theta (X^\lambda_t)^2 = \left( 2\alpha\beta \tilde\EE^\theta \mu(X_t) X^\lambda_t - 2 \lambda \tilde\EE^\theta (X^\lambda_t)^2 + \sigma^2 \right) dt,
\end{align*}
we deduce
\begin{align}
&\lim_{T\rightarrow\infty} \frac{1}{T} \int_0^{T} \alpha\beta \tilde\EE^\theta \mu(X_t) X^\lambda_t dt = \lim_{T\rightarrow\infty} \frac{1}{T} \int_0^{T} \tilde\EE^\theta (\lambda(X^\lambda_t)^2 - \sigma^2/2) dt.
\label{eq.LatentPrice.MainThm.Pf.Le2.muX.to.X2}
\end{align}
Next, we notice that
\begin{align}
&\lambda\left( \lambda\tilde\EE^\theta (X^\lambda_t)^2 - \sigma^2/2\right) = \lambda^2 \tilde\EE^\theta \left[(X^{\lambda}_t)^2 - \left(\sigma \int_0^t e^{-\lambda(t-s)} dW_s\right)^2 \right]\label{eq.LatentPrice.MainThm.Pf.Le2.largeLambda.1}\\
& = \alpha^2 \beta^2  \lambda^2 \tilde\EE^\theta \left(\int_0^t e^{-\lambda(t-s)} \mu(X_s) ds\right)^2
+ 2 \alpha\beta \sigma \lambda^2 \tilde\EE^\theta \int_0^t e^{-\lambda(t-s)} \mu(X_s) ds\, \int_0^t e^{-\lambda(t-s)} dW_s.\nonumber
\end{align}

Let us analyze the first term in the right hand side of \eqref{eq.LatentPrice.MainThm.Pf.Le2.largeLambda.1}:
\begin{align*}
& \left|\lambda^2 \tilde\EE^\theta \left(\int_0^t e^{-\lambda(t-s)} \mu(X_s) ds\right)^2 - \tilde\EE^\theta \mu^2(X_t)\right|\\
&= \left|\lambda^2 \tilde\EE^\theta \left(\int_0^t e^{-\lambda(t-s)} \mu(X_s) ds\right)^2 - \lambda^2 \tilde\EE^\theta \left(\int_0^t e^{-\lambda(t-s)} \mu(X_t) ds\right)^2\right| + O(e^{-\lambda t})\\
&\leq \lambda \int_0^t e^{-\lambda(t-s)} \tilde\EE^\theta |\mu(X_s)-\mu(X_t)| ds + O(e^{-\lambda t})\\
&\leq \lambda \int_0^{\lambda^{-1/2}} e^{-\lambda s} \tilde\EE^\theta |\mu(X_{t-s})-\mu(X_t)| ds + O\left(e^{-\lambda t} + e^{-\sqrt{\lambda}}\right)\\
&\leq \sup_{s\in[0,\lambda^{-1/2}]} \tilde\EE^\theta |\mu(X_{t-s})-\mu(X_t)| + O\left(e^{-\lambda t} + e^{-\sqrt{\lambda}}\right)\\
&\leq \tilde\PP^\theta(X_{t}\,\mmod\,1\notin [\lambda^{-1/8},1-\lambda^{-1/8}]) 
+ \tilde\PP^\theta(\sup_{s\in[0,\lambda^{-1/2}]} |W_s|\geq \lambda^{-1/8})\\
&+ C_1 \lambda \int_0^{\lambda^{-1/2}} e^{-\lambda s} \sqrt{s} ds
+ O\left(e^{-\lambda t} + e^{-\sqrt{\lambda}}\right).
\end{align*}
Note that the last three terms in the right hand side of the above vanish as $\lambda\rightarrow\infty$, uniformly over $t\geq1$.
In addition, Proposition \ref{le:ergod} yields
\begin{align*}
\lim_{\lambda\rightarrow\infty}\lim_{T\rightarrow\infty}\frac{1}{T} \int_0^T \tilde\PP^\theta(X_{t}\,\mmod\,1\notin [\lambda^{-1/8},1-\lambda^{-1/8}]) dt
=0.
\end{align*}
Using the above and Proposition \ref{le:ergod} again, we conclude that\footnote{Strictly speaking, the limits in $T$ in equation \eqref{eq.LatentPrice.MainThm.Pf.Le2.subseqEq} are taken along a subsequence along which the first limit exists. Such a subsequence may depend on $\lambda$, but it always exists due to the uniform boundedness of the associated time-average. To ease the notation, herein and in equation \eqref{eq.LatentPrice.MainThm.Pf.Le2.subseqEq.2}, we avoid the introduction of such a subsequence.}
\begin{align}
& \alpha^2 \beta^2 \lim_{\lambda\rightarrow\infty} \lim_{T\rightarrow\infty} \frac{1}{T} \int_0^T \lambda^2 \tilde\EE^\theta \left(\int_0^t e^{-\lambda(t-s)} \mu(X_s) ds\right)^2 dt
= \lim_{\lambda\rightarrow\infty} \lim_{T\rightarrow\infty} \frac{1}{T} \int_0^T  \tilde\EE^\theta \mu^2(X_t) dt\nonumber\\
&= \alpha^2 \beta^2 \int_0^1 (x-1/2)^2 \chi(x) dx
= \frac{\alpha^2\beta^2}{2\gamma} (1/\psi(\gamma) - 1).\label{eq.LatentPrice.MainThm.Pf.Le2.subseqEq}
\end{align}

\smallskip

Next, we analyze the second term in the right hand side of \eqref{eq.LatentPrice.MainThm.Pf.Le2.largeLambda.1}:
\begin{align}
&\lambda^2 \tilde\EE^\theta \int_0^t e^{-\lambda(t-s)} \mu(X_s) ds\, \int_0^t e^{-\lambda(t-s)} dW_s\nonumber\\
&=\lambda^2 \tilde\EE^\theta \int_0^t e^{-\lambda(t-s)} \mu(X_s) \int_0^s e^{-\lambda(s-u)} dW_u ds\nonumber\\
&=\lambda^3 \tilde\EE^\theta \int_0^t e^{-\lambda(t-s)} \mu(X_s) \int_0^s e^{-\lambda(s-u)}(W_s-W_u) du ds
+ O(\lambda^2 (1+t)^2 e^{-\lambda t})\label{eq.LatentPrice.MainThm.Pf.Le2.mu.times.W}\\
&=\lambda^3 \int_0^t e^{-\lambda(t-s)} \int_0^{1/\sqrt{\lambda}} e^{-\lambda u} \tilde\EE^\theta (\mu(X_s)-\mu(X_{s-u}))(W_s-W_{s-u}) du ds
+ O\left(\lambda^2\left(e^{-\sqrt{\lambda}} + (1+t)^2 e^{-\lambda t}\right)\right).\nonumber
\end{align}
Consider the expectation in the right hand side of the above and approximate $\mu$ with a sequence $\{\mu_n\}$ of $1$-periodic bounded smooth functions converging to $\mu$ in $L^2_{loc}$, as $n\rightarrow\infty$. For convenience, we also assume that $\mu_n(x)=\mu(x)$ for $x\in[1/n,1-1/n]$ and that $\mu_n'(x)\leq1$ for $x\in[-1/n,1/n]$. Then, the dominated convergence theorem yields
\begin{align*}
&\tilde\EE^\theta (\mu(X_s)-\mu(X_{s-u}))(W_s-W_{s-u}) 
= \lim_{n\rightarrow\infty} \tilde\EE^\theta (\mu_n(X_s)-\mu_n(X_{s-u}))(W_s-W_{s-u}).
\end{align*}
In addition, It\^o's formula gives us, for small $u>0$:
\begin{align}
&\tilde\EE^\theta (\mu_n(X_s)-\mu_n(X_{s-u}))(W_s-W_{s-u})\nonumber\\
&=\tilde\EE^\theta (\mu_n(X_s)-\mu_n(X_{s-u}))(X_s-X_{s-u}) + O\left(u \tilde\EE^\theta |\mu_n(X_s)-\mu_n(X_{s-u})| \right)
\label{eq.LatentPrice.MainThm.Pf.Le2.dmu.dW}\\
&= \sigma \int_{s-u}^s \tilde\EE^\theta \mu_n'(X_r) dr
+ \sigma \int_{s-u}^s \tilde\EE^\theta \left[(X_r-X_{s-u}) (\mu_n'(X_r) \mu(X_r) + \frac{\sigma^2}{2}\mu_n''(X_r))\right] dr\nonumber\\
&+ O\left(u \tilde\EE^\theta |\mu_n(X_s)-\mu_n(X_{s-u})| \right).\nonumber
\end{align}

Let us estimate the third term in the right hand side of \eqref{eq.LatentPrice.MainThm.Pf.Le2.dmu.dW}:
\begin{align*}
& \tilde\EE^\theta |\mu_n(X_s)-\mu_n(X_{s-u})|
= \int_0^1 q(s-u,x) \int_\RR |\mu_n(y)-\mu_n(x)| \bar q(x,u,y) dy dx\\
& \leq C_2 \left[\int_0^1 \left(\int_\RR |\mu_n(y)-\mu_n(x)| \bar q(x,u,y) dy\right)^2 dx\right]^{1/2},
\end{align*}
where $q(t,\cdot)$ is the density of $X_t\,\mmod\,1$ and $\bar q(x,t,\cdot)$ is the density of $X_t$ started from $x$.
In the above, we used the fact that $q(t,\cdot)$ is bounded in $L^2([0,1],\RR)$ uniformly over $t\geq1$ (see Proposition \ref{le:ergod}).
In addition, recalling the Gaussian estimates \eqref{eq.Gaussian.est} on $\bar q$, we conclude that the right hand side of the above vanishes as $u\downarrow0$, uniformly over $n\geq1$. Therefore, as $u\downarrow0$,
\begin{align*}
&O\left(u \tilde\EE^\theta |\mu_n(X_s)-\mu_n(X_{s-u})| \right) = o(u),
\end{align*}
uniformly over all $s\geq2$ and $n\geq1$.

Next, we estimate the second term in the right hand side of \eqref{eq.LatentPrice.MainThm.Pf.Le2.dmu.dW}:
\begin{align*}
&\tilde\EE^\theta \left[(X_r-X_{s-u}) (\mu_n'(X_r) \mu(X_r) + \frac{\sigma^2}{2}\mu_n''(X_r))\right]\\
&= \int_0^1 q(s-u,x) \int_\RR \left[(y-x) (\mu_n'(y) \mu(y) + \frac{\sigma^2}{2}\mu_n''(y))\right] \bar q(x,r-s+u,y) dy dx,\\
& \left|\int_\RR (y-x) (\mu_n'(y)-1+1) \mu(y) \bar q(x,r-s+u,y) dy\right|\\
&\leq C_3 \sum_{m\in\ZZ} \sup_{z\in[m-1/n,m+1/n]}|z-x| |\bar q(x,r-s+u,z)|
+ \int_\RR |y-x| |\bar q(x,r-s+u,y)| dy,\\
& \left|\int_\RR (y-x) \mu_n''(y) \bar q(x,r-s+u,y) dy\right|
= \left|\int_\RR (\mu_n'(y)-1+1) \partial_y\left[(y-x) \bar q(x,r-s+u,y)\right] dy\right|\\
&\leq C_3 \sum_{m\in\ZZ} \sup_{z\in[m-1/n,m+1/n]}\left[\bar q(x,r-s+u,z) + |z-x| |\partial_z \bar q(x,r-s+u,z)|\right].
\end{align*}
Using the above inequalities and the Gaussian estimates \eqref{eq.Gaussian.est}, we obtain:
\begin{align*}
&\left|\int_\RR (y-x) (\mu_n'(y)-1+1) \mu(y) \bar q(x,r-s+u,y) dy\right|
\leq C_4 (r-s+u)^{-1/2},\\
&\phantom{??????????????????????????} \forall\, r-s+u\in(0,1],\,\,x\in[0,1],\,\,n\geq1,\\
& \limsup_{n\rightarrow\infty}\left|\int_\RR (y-x) (\mu_n'(y)-1+1) \mu(y) \bar q(x,r-s+u,y) dy\right|
\leq C_5 (r-s+u)^{1/2},\\
&\phantom{??????????????????????????} \forall\, r-s+u\in(0,1],\,\,x\in(0,1).
\end{align*}
The above and the dominated convergence theorem yield
\begin{align*}
&\left|\tilde\EE^\theta \left[(X_r-X_{s-u}) (\mu_n'(X_r) \mu(X_r) + \frac{\sigma^2}{2}\mu_n''(X_r))\right]\right|
\leq C_4 (r-s+u)^{-1/2},\\
&\phantom{??????????????????????????} \forall\, r-s+u\in(0,1],\,\,n\geq1,\\
&\limsup_{n\rightarrow\infty} \left|\tilde\EE^\theta \left[(X_r-X_{s-u}) (\mu_n'(X_r) \mu(X_r) + \frac{\sigma^2}{2}\mu_n''(X_r))\right]\right|
\leq C_5 (r-s+u)^{1/2},\\
&\phantom{??????????????????????????} \forall\, r-s+u\in(0,1],\\
&\limsup_{n\rightarrow\infty} \left|\int_{s-u}^s \tilde\EE^\theta \left[(X_r-X_{s-u}) (\mu_n'(X_r) \mu(X_r) + \frac{\sigma^2}{2}\mu_n''(X_r))\right] dr\right| \leq C_6 u^{3/2},\quad \forall\, u\in(0,1].
\end{align*}

Next, we consider the first term in the right hand side of \eqref{eq.LatentPrice.MainThm.Pf.Le2.dmu.dW}:
\begin{align*}
&\sigma \int_{s-u}^s \tilde\EE^\theta \mu_n'(X_r) dr
=\sigma \int_{s-u}^s \int_0^1 q(r,x) \mu_n'(x)dx dr
=-\sigma \int_{s-u}^s \int_0^1 \partial_x q(r,x) \mu_n(x)dx dr\\
&\rightarrow -\sigma \int_{s-u}^s \int_0^1 \partial_x q(r,x) \mu(x)dx dr,
\end{align*}
as $n\rightarrow\infty$, where we used the fact that $\int_{s-u}^{s}\|q(r,\cdot)\|_{H^1([0,1],\RR)} dr<\infty$ (see the second part of Proposition \ref{le:ergod}). 
Collecting the above, recalling \eqref{eq.LatentPrice.MainThm.Pf.Le2.dmu.dW}, and using the second part of Proposition \ref{le:ergod} once more, we obtain:
\begin{align*}
&\tilde\EE^\theta (\mu(X_s)-\mu(X_{s-u}))(W_s-W_{s-u}) 
= \lim_{n\rightarrow\infty} \tilde\EE^\theta (\mu_n(X_s)-\mu_n(X_{s-u}))(W_s-W_{s-u})\\
& = \sigma u \left(-\int_0^1 \partial_x \chi(x) \mu(x)dx + g(s,u) + h(s,u)\right) = \sigma u\left(1 - \chi(1)+ g(s,u) + h(s,u)\right),
\end{align*}
for some bounded functions $g,h:\RR_+\times(0,1]\rightarrow\RR$, s.t. $\lim_{s\rightarrow\infty}\sup_{u\in(0,1]}|g(s,u)|=0=\lim_{u\downarrow0}\sup_{s\in\RR_+}|h(s,u)|$.
Recalling \eqref{eq.LatentPrice.MainThm.Pf.Le2.mu.times.W}, we deduce that
\begin{align*}
&\lambda^2 \tilde\EE^\theta \int_0^t e^{-\lambda(t-s)} \mu(X_s) ds\, \int_0^t e^{-\lambda(t-s)} dW_s\\
&=\lambda^3 \int_0^t e^{-\lambda(t-s)} \int_0^{1/\sqrt{\lambda}} e^{-\lambda u} \left[ \sigma u(1 - \chi(1)) + u g(s,u) + u h(s,u) \right] du ds\\
&+ O\left(\lambda^2\left(e^{-\sqrt{\lambda}} + (1+t)^2 e^{-\lambda t}\right)\right)\\
& = \sigma \lambda \int_0^t e^{-\lambda(t-s)} \int_0^{\sqrt{\lambda}} e^{-u} u \left[ 1 - \chi(1) + g(s,u/\lambda) + h(s,u/\lambda) \right] du ds\\
&+ O\left(\lambda^2\left(e^{-\sqrt{\lambda}} + (1+t)^2 e^{-\lambda t}\right)\right)
 = \sigma (1 - \chi(1)) \int_0^{\sqrt{\lambda}} e^{-u} u du\, \lambda \int_0^t e^{-\lambda(t-s)} ds\\
&+\sigma \lambda \int_{t/2}^t e^{-\lambda(t-s)} \int_0^{\sqrt{\lambda}} e^{-u} u \left[ g(s,u/\lambda) + h(s,u/\lambda) \right] du ds
+ O\left((1+\lambda^2)\left(e^{-\sqrt{\lambda}} + (1+t)^2 e^{-\lambda t/2}\right)\right).
\end{align*}
As $t\rightarrow\infty$, the upper limit of the absolute value of the second term in the right hand side of the above is bounded by
\begin{align*}
&\lim_{t\rightarrow\infty}\sigma \left(\sup_{s\in[t/2,t],\,u\in(0,1]}|g(s,u)| + \sup_{s\in\RR_+,\,u\in(0,1/\sqrt{\lambda}]}|h(s,u)|\right)\lambda \int_{t/2}^t e^{-\lambda(t-s)} ds
= \sigma \sup_{s\in\RR_+,\,u\in(0,1/\sqrt{\lambda}]}|h(s,u)|,
\end{align*}
which vanishes as $\lambda\rightarrow\infty$.

Finally, we can analyze the second term in the right hand side of \eqref{eq.LatentPrice.MainThm.Pf.Le2.largeLambda.1}:
\begin{align}
&2 \alpha\beta \sigma \lim_{\lambda\rightarrow\infty} \lim_{T\rightarrow\infty} \frac{1}{T} \int_0^T \lambda^2 \tilde\EE^\theta \int_0^t e^{-\lambda(t-s)} \mu(X_s) ds\, \int_0^t e^{-\lambda(t-s)} dW_s\, dt\label{eq.LatentPrice.MainThm.Pf.Le2.subseqEq.2}\\
& =2 \alpha\beta \sigma \lim_{\lambda\rightarrow\infty} \sigma (1 - \chi(1)) \int_0^{\sqrt{\lambda}} e^{-u} u du
= 2 \alpha\beta \sigma^2 (1 - \chi(1)) = \frac{2 \alpha^2 \beta^2}{\gamma} (1 - 1/\psi(\gamma)).
\nonumber
\end{align}

\smallskip

Collecting \eqref{eq.LatentPrice.MainThm.Pf.Le2.muX.to.X2}, \eqref{eq.LatentPrice.MainThm.Pf.Le2.largeLambda.1}, \eqref{eq.LatentPrice.MainThm.Pf.Le2.subseqEq} and \eqref{eq.LatentPrice.MainThm.Pf.Le2.subseqEq.2}, we conclude that
\begin{align*}
&\frac{\beta}{\alpha}\lim_{\lambda\rightarrow\infty}\lambda \lim_{T\rightarrow\infty} \frac{1}{T} \int_0^T \tilde\EE^\theta \mu(X_t) X^\lambda_t dt
= \frac{1}{\alpha^2}\lim_{\lambda\rightarrow\infty} \lim_{T\rightarrow\infty} \frac{1}{T} \int_0^{T} \lambda\tilde\EE^\theta (\lambda(X^\lambda_t)^2 - \sigma^2/2) dt\\
& = \frac{1}{\alpha^2}\lim_{\lambda\rightarrow\infty} \lim_{T\rightarrow\infty} \frac{1}{T} \int_0^{T} \alpha^2 \beta^2  \lambda^2 \tilde\EE^\theta \left(\int_0^t e^{-\lambda(t-s)} \mu(X_s) ds\right)^2\,dt\\
&+ \frac{1}{\alpha^2}\lim_{\lambda\rightarrow\infty}\lambda \lim_{T\rightarrow\infty} \frac{1}{T} \int_0^T 2 \alpha\beta \sigma \lambda^2 \tilde\EE^\theta \int_0^t e^{-\lambda(t-s)} \mu(X_s) ds\, \int_0^t e^{-\lambda(t-s)} dW_s\,dt\\
& = \frac{\beta^2}{2\gamma} (1/\psi(\gamma) - 1)
+ \frac{2 \beta^2}{\gamma} (1 - 1/\psi(\gamma)) = \frac{3 \beta^2}{2\gamma} (1 - 1/\psi(\gamma)).
\end{align*}
Repeating the above for $\theta'$ in place of $\theta$ and recalling \eqref{eq.LatentPrice.MainThm.Pf.Le2.eq1}, we obtain:
\begin{align}
& \frac{\beta^2}{\gamma(\theta)} (1 - 1/\psi(\gamma(\theta)))
= \frac{(\beta')^2}{\gamma(\theta')} (1 - 1/\psi(\gamma(\theta'))),
\label{eq.Le2.Step1.main}
\end{align}
which is the main result of Step 1.

\medskip

{\bf Step 2}. Herein, we analyze $\lambda\downarrow0$.
It is easy to see that, for any finite $t>0$, we have
\begin{align}\label{eq.mainThm.pf.exp.Xlambda.via.X}
\lim_{\lambda\downarrow0} \tilde\EE^\theta \mu(X_t) X^\lambda_t
=\tilde\EE^\theta \mu(X_t) X_t,\quad
\lim_{\lambda\downarrow0} \tilde\EE^{\theta'} \mu(X_t) X^\lambda_t
=\tilde\EE^{\theta'} \mu(X_t) X_t.
\end{align}

\smallskip

The following lemma shows that we can interchange the limits in $t\rightarrow\infty$ and in $\lambda\downarrow0$ in the above expressions.

\begin{lemma}\label{le:mainThm.pf.unif.Conv}
For any $\theta\in\Theta$, we have
\begin{align*}
\lim_{\lambda \downarrow 0} \sup_{t\geq1} \left| \tilde\EE^\theta \mu(X_t) X^\lambda_t - \tilde\EE^{\theta} \mu(X_t) X_t \right|=0
\end{align*}
\end{lemma}
\begin{proof}
\begin{align*}
& \tilde\EE^\theta \mu(X_t) X^\lambda_t - \tilde\EE^{\theta} \mu(X_t) X_t
= -\lambda \int_0^t e^{-\lambda(t-s)} \tilde\EE^\theta \mu(X_t) X_s ds
= -\lambda \int_0^t e^{-\lambda(t-s)} \tilde\EE^\theta \left( \tilde\EE^\theta \left[ \mu(X_t)\,\vert\, X_s \right] X_s\right) ds\\
&=: -\lambda \int_0^t e^{-\lambda(t-s)} \tilde\EE^\theta \left( \tilde g(t-s,X_s) X_s\right) ds.
\end{align*}
Then, to prove the statement of the lemma it suffices to show that
$$
\lim_{u\rightarrow\infty}\sup_{s\geq1} |\tilde\EE^\theta\left( \tilde g(u,X_s) X_s\right)| = 0.
$$
To this end we notice that $\tilde g(u,\cdot)$ is 1-periodic and that $\tilde g(u,x)=-\tilde g(u,1-x)$, $\tilde g(u,0)=\tilde g(u,1)=0$, and recall from Lemma \ref{le:ergod} that 
\begin{align}
\lim_{u\rightarrow\infty}\sup_{x\in[0,1]}|\tilde g(u,x)|=0.\label{eq.mainThm.pf.Le2.Step2.Le.eq1}
\end{align}

\smallskip

Next, we recall the scale transform of $X$:
\begin{align*}
&Z_t=S(X_t),\quad S(x):=\int_0^x \exp\left(-2\int_{1/2}^y \mu(z)/\sigma^2 dz\right) dy
=\int_0^x \exp\left(-\mu^2(y)/\sigma^2\right) dy,\\
& dZ_t=\exp\left(-2\int_{1/2}^{X_t} \mu(z)/\sigma^2(z) dz\right) dX_t 
- \exp\left(-2\int_{1/2}^{X_t} \mu(z)/\sigma^2(z) dz\right) \mu(X_t) dt\\
& = \tilde\sigma(Z_t) dW_t,
\quad \tilde\sigma(z)= S'(S^{-1}(z)) \sigma.
\end{align*}
Notice that
\begin{align*}
\tilde g(u,X_s) X_s = g(u,Z_s) S^{-1}(Z_s),\quad g(u,z):=\tilde g(u, S(z)),
\end{align*}
where $g(u,\cdot)$ satisfies the same properties as $\tilde g(u,\cdot)$, except that its period is $p=S(1)$.
Thus, our goal is to show that
\begin{align*}
\lim_{u\rightarrow\infty}\sup_{s\geq1} |\tilde\EE^\gamma\left( g(u,Z_s) S^{-1}(Z_s)\right)| = 0.
\end{align*}

To show the above, we recall some classical function-analytic results.
Denote by $H^2_0([0,p],\RR)$ the Sobolev space of functions $v$ with square integrable first and second derivatives and such that $v(x)=-v(p-x)$, $v(0)=v(p)=0$. This space is endowed with its natural Sobolev norm. 
The generator
\begin{align*}
\mathcal{L} := \frac{1}{2} \tilde\sigma^2(x) \partial^2_{xx}
\end{align*}
of $Z$ is well defined on $H^2_0([0,p],\RR)$
which is compactly and densely embedded in the Hilbert space $L^2([0,p],\tilde\sigma^{-2},\RR)$. It is easy to see that $\mathcal{L}$ is self-adjoint w.r.t. the cross-product of the latter space.
It is also well known (and can be verified directly) that the inverse $\mathcal{L}^{-1}$ is well defined on $L^2([0,p],\tilde\sigma^{-2},\RR)$ (notice that $\tilde\sigma(x)=\tilde\sigma(1-x)$). Then, $\mathcal{L}^{-1}$ is compact and the Riesz-Schauder theorem implies that there exists an orthonormal basis $\{v_i\}_{i=0}^\infty$ in $L^2([0,p],\tilde\sigma^{-2},\RR)$ and a sequence $\{-\lambda_i\}$, s.t. $\{|\lambda_i|\}$ is non-decreasing and converges to $\infty$, $\mathcal{L} v_i = -\lambda_i v_i$ (in particular, $v_i\in H^2_0([0,p],\RR)$). It is easy to deduce from the definition of $\mathcal{L}$ that $\lambda_i>0$.

\smallskip

As $g(u,\cdot)\in L^2([0,p],\tilde\sigma^{-2},\RR)$, we have:
\begin{align*}
& g(u,z) = \sum_{i=1}^\infty c_i(u) v_i(z),\quad \tilde\EE^\theta g(u,Z_t) S^{-1}(Z_t) = \sum_{i=1}^\infty c_i(u) \tilde\EE^\theta v_i(Z_t\,\mmod\,p) S^{-1}(Z_t).
\end{align*}
Extending $v_i$ and $\tilde\sigma$ to the $p$-periodic functions on $\RR$, we recall the definition of $H^2_0$ to conclude that the extended $v_i$ has a weak second derivative that is locally square integrable. Then, using It\^o's formula, we obtain:
\begin{align*}
& d\tilde\EE^\theta v_i(Z_t) S^{-1}(Z_t) =\tilde\EE^\theta \left( \frac{1}{2}\tilde\sigma^2(Z_t)\partial^2_{zz} \left[v_i(z) S^{-1}(z)\right](Z_t) \right) dt\\
& =\tilde\EE^\theta \left(\frac{1}{2}\tilde\sigma^2(Z_t) v_i''(Z_t) S^{-1}(Z_t) + \tilde\sigma^2(Z_t) v_i'(Z_t) (S^{-1})'(Z_t)
+ \frac{1}{2}\tilde\sigma^2(Z_t) v_i(Z_t) (S^{-1})''(Z_t) \right) dt\\
& = \left[-\lambda_i\tilde\EE^\theta v_i(Z_t) S^{-1}(Z_t) + \tilde\EE^\theta \left(\tilde\sigma^2(Z_t) v_i'(Z_t) (S^{-1})'(Z_t)
+ \frac{1}{2}\tilde\sigma^2(Z_t) v_i(Z_t) (S^{-1})''(Z_t) \right)\right] dt.
\end{align*}
Denote by $\tilde q(s,\cdot)$ the density of $Z_s\,\mmod\,p$. It is shown in the proof of Proposition \ref{le:ergod} that $\tilde q\in L^2_{loc}((0,\infty),H^{1}([0,p],\RR))$ and $\tilde q(s,0)=\tilde q(s,p)$.
In addition, we notice that $(S^{-1})'$, $(S^{-1})''$ are $p$-periodic and that $S^{-1}(0)=S^{-1}(p)$, $\tilde\sigma(0)=\tilde\sigma(p)$.
Using these observations, we continue:
\begin{align*}
& \tilde\EE^\theta v_i(Z_t) S^{-1}(Z_t) = e^{-\lambda_i (t-1)}\tilde\EE^\theta v_i(Z_1) S^{-1}(Z_1)\\
&+ \int_1^t e^{-\lambda_i(t-s)} \tilde\EE^\theta \left(\tilde\sigma^2(Z_s) v_i'(Z_s) (S^{-1})'(Z_s)
+ \frac{1}{2}\tilde\sigma^2(Z_s) v_i(Z_s) (S^{-1})''(Z_s) \right) ds\\
& = e^{-\lambda_i (t-1)}\tilde\EE^\theta v_i(Z_1) S^{-1}(Z_1)\\
&+ \int_1^t e^{-\lambda_i(t-s)} \int_0^p v_i(x)\left[\frac{1}{2}\tilde\sigma^2(x) (S^{-1})''(x)\tilde q(s,x)
- \partial_x\left(\tilde\sigma^2(x) (S^{-1})'(x) \tilde q(s,x)\right)\right] dx ds.
\end{align*}
Denote
\begin{align*}
u(s,x):= \frac{1}{2}\tilde\sigma^4(x) (S^{-1})''(x)\tilde q(s,x)
- \tilde\sigma^2(x)\partial_x\left(\tilde\sigma^2(x) (S^{-1})'(x) \tilde q(s,x)\right),\quad x\in[0,p].
\end{align*}
The second part of Proposition \ref{le:ergod} implies that $\tilde q(s,\cdot)$ is bounded in $L^2$ uniformly over all $s\geq1$.
It is also shown in the proof of Proposition \ref{le:ergod} (recall \eqref{eq.Le1.EnergyEst} and notice that $\partial_y\left[\tilde\sigma^2(y) g(r,y)\right]=\partial_y\left[\tilde\sigma^2(y) \tilde q(r,y)\right]$) that 
\begin{align*}
\sup_{t\geq1}\int_1^t \int_0^p \left(\partial_y\left[\tilde\sigma^2(y) \tilde q(r,y)\right]\right)^2 dy ds <\infty.
\end{align*}
Then, for any $\lambda>0$,
\begin{align*}
&\sup_{t\geq1}\int_1^t e^{-\lambda(t-s)} \|u(s,\cdot)\|_{L^2([0,p],\tilde\sigma^{-2},\RR)} ds<\infty.
\end{align*}
In addition, denoting by $\hat q(t,\cdot)$ the density of $Z_t$ and using the Gaussian estimates \eqref{eq.Gaussian.est}, we obtain:
\begin{align*}
\sum_{n\in\ZZ} \| \hat q(1,np+\cdot) S^{-1}(np+\cdot) \tilde\sigma^{2}(x)\|_{L^2([0,p],\tilde\sigma^{-2},\RR)} <\infty.
\end{align*}
Using the above observations, we conclude:
\begin{align*}
& \left|\tilde\EE^\theta g(u,Z_t) S^{-1}(Z_t) \right|
\leq \left|\sum_{i=1}^\infty c_i(u) e^{-\lambda_i (t-1)} \sum_{n\in\ZZ} \int_0^p v_i(x) S^{-1}(x+np) \hat q(1,np+x)dx\right|\\
&+\left|\int_1^t \sum_{i=1}^\infty c_i(u) e^{-\lambda_i(t-s)} \int_0^p v_i(x) u(s,x) dx ds\right|\\
&\leq \sum_{n\in\ZZ} \left(\sum_{i=1}^\infty c^2_i(u) \right)^{1/2} \left(\sum_{i=1}^\infty \left( \int_0^p v_i(x)S^{-1}(x+np) \hat q(1,np+x) \tilde\sigma^{2}(x) \tilde\sigma^{-2}(x)dx\right)^2\right)^{1/2}\\
& + \int_1^t e^{-\lambda_i (t-s)} \left(\sum_{i=1}^\infty c^2_i(u) \right)^{1/2} \left(\sum_{i=1}^\infty \left(\int_0^p v_i(x) u(s,x) \tilde\sigma^{-2}(x)dx\right)^2 \right)^{1/2}\\
&\leq \left(\sum_{i=1}^\infty c^2_i(u) \right)^{1/2} \sum_{n\in\ZZ} \| \hat q(1,np+\cdot) S^{-1}(np+\cdot) \tilde\sigma^{2}(x)\|_{L^2([0,p],\tilde\sigma^{-2},\RR)}\\
&+ \int_1^t e^{-\lambda_0(t-s)} \left(\sum_{i=1}^\infty c^2_i(u) \right)^{1/2} \|u(s,\cdot)\|_{L^2([0,p],\tilde\sigma^{-2},\RR)} ds\\
& \leq C \|g(u,\cdot)\|_{L^2([0,p],\tilde\sigma^{-2},\RR)},\quad \forall\, t\geq1.
\end{align*}
It remains to recall that the right hand side of the above converges to zero as $u\rightarrow\infty$ (see \eqref{eq.mainThm.pf.Le2.Step2.Le.eq1}).
\qed
\end{proof}

\medskip

Next, we compute the limits of the right hand sides in \eqref{eq.mainThm.pf.exp.Xlambda.via.X}, as $t\rightarrow\infty$.
We claim that
\begin{align}\label{eq.mainThm.pf.exp.XmuX.lim.t}
\lim_{t\rightarrow\infty} \frac{1}{t}\int_0^t \tilde\EE^\theta \mu(X_s) X_s ds
= \frac{\Sigma(\gamma)-\sigma^2}{2\alpha\beta}.
\end{align}
Indeed, applying It\^o's lemma to $X^2_t$, we obtain
\begin{align*}
d\tilde\EE^\theta X^2_t = \left( 2\alpha\beta\tilde\EE^\theta \mu(X_t) X_t + \sigma^2 \right) dt.
\end{align*}
In addition, it follows from the proof of Proposition \ref{prop:Sigma} that
\begin{align*}
\lim_{t\rightarrow\infty} \frac{1}{t} \tilde\EE^\theta X^2_t = \Sigma(\theta),
\end{align*}
which altogether yields \eqref{eq.mainThm.pf.exp.XmuX.lim.t}.

Combining Lemma \ref{le:mainThm.pf.unif.Conv} and \eqref{eq.mainThm.pf.exp.XmuX.lim.t}, we obtain
\begin{align*}
&\lim_{\lambda\downarrow0}\lim_{t\rightarrow\infty}\frac{1}{t}\int_0^t \tilde \EE^\theta \mu(X_s) X^\lambda_s ds = \frac{\Sigma(\theta)-\sigma^2}{2\alpha\beta}.
\end{align*}
Analogous equality holds for $\theta'$ in place of $\theta$.
Then, using \eqref{eq.LatentPrice.MainThm.Pf.Le2.eq1}, we obtain the main result of Step 2:
\begin{align}
& \frac{\sigma^2 - \Sigma(\theta)}{\alpha^2} = \frac{(\sigma')^2 - \Sigma(\theta')}{(\alpha')^2}>0.
\label{eq.Le2.Step2.main}
\end{align}

\medskip

We denote $\Sigma_0:=\Sigma(\theta)=\Sigma(\theta')$ and treat it as a constant. Then, treating the right hand sides of \eqref{eq.Le2.Step1.main} and \eqref{eq.Le2.Step2.main} as given constants, we obtain
\begin{align*}
&\beta^2 = C_7\frac{\gamma(\theta)}{1 - 1/\psi(\gamma(\theta))},
\quad \alpha^2 = C_8 (\sigma^2/\Sigma_0-1)=C_8 (\phi(\gamma(\theta))-1).
\end{align*}
Recalling that
\begin{align*}
\gamma^2=\alpha^2\beta^2/\sigma^4 = \alpha^2\beta^2/(\Sigma_0^2\phi^2(\gamma)),
\end{align*}
we obtain the equation
\begin{align}
\frac{1-1/\phi(\gamma)}{\gamma(1/\psi(\gamma)-1) \phi(\gamma)} = C_9.
\label{eq.LatentPrice.MainThm.Pf.Le2.Pf.last}
\end{align}
Recalling that the constant $C_9$ is in fact equal to the left hand side of \eqref{eq.LatentPrice.MainThm.Pf.Le2.Pf.last} with $\theta$ replaced by $\theta'$, we conclude that $\gamma=\gamma(\theta')$ solves the above equation. Assumption \ref{ass:1} implies that such solution is unique and, hence, $\gamma(\theta)=\gamma(\theta')$. The latter, along with $\sigma^2=\Sigma_0\phi(\gamma(\theta))=\Sigma_0\phi(\gamma(\theta'))=(\sigma')^2$ and \eqref{eq.Le2.Step1.main}, \eqref{eq.Le2.Step2.main}, yields $\theta=\theta'$.
\qed
\end{proof}

\section{Implementation and empirical analysis}
\label{se:implem}


Let us discuss a numerical algorithm for approximating $L_T(\theta)$ in the financial example of Section \ref{se:latentprice}. Note that the observed data is always discrete, even if it is observed at a very high frequency. Therefore, we aim to construct an approximation method that converges to the true $L_T(\theta)$ as the observation frequency increases. 
The main component in the computation of $L_T(\theta)$ (recall \eqref{eq.LatentPrice.L.def}) is 
\begin{align}
\mu^\theta_t = \tilde\EE^\theta\left(\mu(X_t)\,\vert\,\mathcal{F}^Y_t\right)
=  \tilde\EE^\theta\left[ \mu\left(\kappa Y_t + \tilde X_t\right)\,\vert\,\mathcal{F}^Y_t\right],
\label{eq.Implem.mu.via.tildeX.1}
\end{align}
where $\tilde X_t:=X_t - \kappa Y_t$ and $\kappa:=\frac{\sigma^2}{\alpha\bar\sigma^2}$.
It is easy to see that
\begin{align}
d\tilde X_t = \beta \frac{\alpha^2\bar\sigma^2 - \sigma^2}{\alpha\bar\sigma^2} \mu(\kappa Y_t + \tilde X_t)dt
+ \sigma \frac{\alpha^2\bar\sigma^2 - \sigma^2}{\alpha^2\bar\sigma^2} dB_t
- \sigma^2\frac{\sqrt{\alpha^2\bar\sigma^2-\sigma^2}}{\alpha^2\bar\sigma^2} d\tilde W_t,
\label{eq.Implem.dTildeX} 
\end{align}
that the quadratic covariation $\langle \tilde X,Y \rangle$ between $\tilde X$ and $Y$ is zero,
and that $\langle \tilde X, \tilde X \rangle_t = t \sigma^2 \left(1-\frac{\sigma^2}{\alpha^2\bar\sigma^2}\right)$.

Then, the standard results in stochastic filtering (see \cite[Section~5.6.1]{LototskyRozovsky2017Book}, and \cite{BainCrisan} for more) imply that the conditional distribution of $\tilde X_t$ given $\mathcal{F}^Y_t$ has a density that can be computed by normalizing the function $u(t,\cdot)$ which solves the Zakai equation:
\begin{align*}
&du = \left[\frac{1}{2} A^2 \partial^2_{xx} u - C \partial_x\left(\mu(x+\kappa Y_t) u\right) \right] dt
+ \frac{\beta}{\bar{\sigma}^2}\mu(x+\kappa Y_t) u dY_t,\\
& A^2:= \sigma^2 \frac{\alpha^2\bar\sigma^2-\sigma^2}{\alpha^2\bar\sigma^2},
\quad C:= \beta \frac{\alpha^2\bar\sigma^2 - \sigma^2}{\alpha\bar\sigma^2}.
\end{align*}
Note that, to compute $L_T$, it suffices to have the conditional density of $\tilde X_t\,\mmod\,1$ given $\mathcal{F}^Y_t$. Using the 1-periodicity of the coefficients in the above Zakai equation, we easily deduce that the latter density satisfies the same equation, with periodic boundary conditions at $x=0,1$. With a slight abuse of notation, we denote by $u(t,\cdot)$ the conditional density of $\tilde X_t\,\mmod\,1$ given $\mathcal{F}^Y_t$.

\smallskip

Next, we need to approximate $u$ numerically. 
To avoid stability issues and keep the implementation simple, we use the operator splitting method to approximate $u$. Namely, we consider two auxiliary equations
\begin{align}
&du = \left[\frac{1}{2} A^2 \partial^2_{xx} u - C \partial_x\left(\mu(x+\kappa Y_t) u\right) \right] dt,\label{eq.Implem.SPDEsplit.1}\\
&du = \frac{\beta}{\bar{\sigma}^2}\mu(x+\kappa Y_t) u dY_t,\label{eq.Implem.SPDEsplit.2}
\end{align}
equipped with periodic boundary conditions.
Equation \eqref{eq.Implem.SPDEsplit.2} can be solved for each $x$ separately via the operator:
\begin{align*}
G^{\Delta t}_2[u_0](x):= u_0(x) \exp\left[ -\frac{\beta^2}{2\bar\sigma^2} \int_0^{\Delta t} \mu^2(x+\kappa Y_t) dt
+ \frac{\beta}{\bar{\sigma}^2} \int_0^{\Delta t} \mu(x+\kappa Y_t) dY_t \right],
\end{align*}
where we fix $\Delta t=1$ (with the unit of time being $1$ second) and approximate the above integrals via the associated single-term Riemann sums.

To approximate the solution to \eqref{eq.Implem.SPDEsplit.1}, we notice that it is the Fokker-Planck equation for a diffusion $Z$ on the unit-length circle with the generator
\begin{align*}
\mathcal{L}:= \frac{1}{2} A^2 \partial^2_{xx} + C\mu(\cdot+\kappa Y_t)\partial_x.
\end{align*}
To approximate the transition kernel of $Z$, we first denote by $K^{\Delta t}(x,y)=\tilde K^{\Delta t}(|x-y|)$ the transition kernel of Brownian motion scaled by $A$ on the unit-length circle (i.e., the transition kernel of $(A W)\,\mmod\,1$, where $W$ is a Brownian motion) over a time interval of length $\Delta t$. It is easy to compute this kernel numerically by summing up the values of the associated Gaussian kernel. Then, we approximate the transition kernel of $Z$ via
\begin{align*}
\Gamma^{\Delta t, Y_t}(z,x):= K^{\Delta t}\left(z,\left(x-C\mu(x+\kappa Y_t)\Delta t\right)\,\mmod\,1\right)
=\tilde K^{\Delta t}\left(\left|z-\left(x-C\mu(x+\kappa Y_t)\Delta t\right)\,\mmod\,1\right|\right),
\end{align*}
and approximate the solution to \eqref{eq.Implem.SPDEsplit.1} via
\begin{align*}
&G^{\Delta t, Y_t}_1[u_0](x):= \int_0^1 \tilde K^{\Delta t}\left(\left|z-b(x,\Delta t,Y_t)\right|\right) u_0\left(z\right) dz,\\
& b(x,\Delta t,Y_t):=\eta + e^{-C\Delta t} (x-\eta),\quad \eta:= (1/2-\kappa Y_t)\,\mmod\,1.
\end{align*}

The proposed approximation of $u$ is then given by
\begin{align*}
\left(\prod_{i=T/\Delta t - 1}^0 G^{\Delta t}_2 G^{\Delta t,Y_{i\Delta t}}_1\right)[u_0],
\end{align*}
where $u_0$ is the density of $X_0\,\mmod\,1$ and where we use an equidistant grid with mesh $\Delta x:=0.01$ (with the unit of space being $\$0.01$) to discretize the space domain $[0,1]$.
Recalling \eqref{eq.Implem.mu.via.tildeX.2}, we obtain the following approximation of $\mu^\theta_t$:
\begin{align}
\mu^\theta_t =  \tilde\EE^\theta\left[ \mu\left(\kappa Y_t + \tilde X_t\right)\,\vert\,\mathcal{F}^Y_t\right]
= \int_0^1 \mu(\kappa Y_t + x) u(t,x) dx,
\label{eq.Implem.mu.via.tildeX.2}
\end{align}
and compute $L_T(\theta)$ via replacing the integrals in \eqref{eq.LatentPrice.L.def} by the associated Riemann sums.
Once $L_T(\cdot)$ is computed, we find its maximizer by an exhaustive search over the chosen grid for $\theta$.

\medskip

The computation of $\hat \Sigma$ is straightforward as given by Proposition \ref{prop:Sigma}. The constant $\bar\sigma$ is estimated via \eqref{eq.sigmaBar.est} by approximating the quadratic variation by a sum of squared increments of the order flow over $1$-second time intervals. The unknown $\beta$ is computed from the equation $\Sigma(\alpha,\beta,\sigma^2)=\hat\Sigma$ for each pair $(\alpha,\sigma^2)$ by a simple inversion (it is easy to see that $\Sigma(\alpha,\cdot,\sigma^2)$ is strictly monotone). The set of values for $(\alpha,\sigma^2)$ is defined by an equidistant rectangular grid for $(\alpha^2,\sigma^2)$ spanning the corresponding rectangle with $100$ points in each dimension (i.e., $10^4$ points in total).

\subsection{Empirical results}
\label{subse:empirical}

The proposed algorithm is tested using the price and orders data for three NASDAQ stocks\footnote{We thank S. Jaimungal for providing the financial data.} with tickers INTC, MSFT and MU, over the period Nov 3-26, 2014. Each computation of $L_T$ and of its maximizer is performed using the data for a single ticker over a single day, from 10:30am to 3pm Eastern Standard Time. Thus, we conduct $24$ experiments for each ticker. The choice of these specific assets is motivated by (i) their high liquidity and (ii) the fact that their bid-ask spread is equal to one tick ($\$0.01$) most of the time (recall that the latter is a necessary condition for the proposed latent price model to make sense).

\smallskip

The graphs of $\hat\Sigma$ and of estimated $\bar\sigma^2$ are given in Figures \ref{fig:Sigma.INTC}--\ref{fig:Sigma.MU}. The examples of normalized log-likelihood (i.e., $\frac{1}{T}\log L_T$) as a function of $(\alpha^2,\sigma^2)$ (recall that $\beta$ is uniquely determined by $(\alpha^2,\sigma^2)$ via the condition $\Sigma(\theta)=\hat\Sigma$, provided $\sigma^2>\hat\Sigma$) are given in Figures \ref{fig:L.INTC}--\ref{fig:L.MU}. For a better visualization, we provide cross-sectional plots, where each curve corresponds to a fixed $\alpha^2$ and varying $\sigma^2$ (right sides of Figures \ref{fig:L.INTC}--\ref{fig:L.MU}), and vice versa (left sides of Figures \ref{fig:L.INTC}--\ref{fig:L.MU}). The graphs of the resulting MLEs $(\hat\alpha^2,\hat\beta,\hat\sigma^2)$ are given in Figures \ref{fig:MLE.INTC}--\ref{fig:MLE.MU}.

\smallskip

Recall that the latent price model \eqref{eq.mainDyn.def}--\eqref{eq.A.def} contains another unknown parameter $\epsilon$, which determines when the bid-ask spread widens from one to two ticks (i.e., when $X$ is closer than $\epsilon$ to an integer). The estimation procedure was designed so that this parameter played no role in it. Nevertheless, once all the pother parameters are estimated, using the ergodicity of $X\,\mmod\,1$, one can easily construct a consistent estimator of $\epsilon$:
\begin{align*}
\hat\epsilon := \text{argmin}_{\varepsilon\in(0,1/2)}\left| \frac{1}{T} \int_0^T \bone_{[2,\infty)}(P^a_t-P^b_t) dt
\,-\, \int_{[0,1]\setminus [\epsilon,1-\epsilon]} \chi(x) dx \right|,
\end{align*}
where $\chi$ is the stationary density of $X\,\mmod\,1$ (see \eqref{eq.LatentPrice.mainThm.Pf.chi.def}).
These estimators are given in Figure \ref{fig:epsilon}. Not surprisingly, they are quite small, as the bid-ask spread of the chosen stocks remains at one tick most of the time.

\smallskip

In fact, the estimated values of $\epsilon$ are made even smaller by the fact that $\chi$ is U-shaped: see Figure \ref{fig:chi}. The latter property is very important and is used in \cite{N} to explain the concavity of expected price impact of meta-orders. Namely, it is shown in \cite{N} that the deviation of the wings of $\chi$ from one determines the relative change in the slope of the expected price impact curve, viewed as a function of executed volume, for a large meta-order. Figure \ref{fig:chi} shows that this relative difference can be significant, providing empirical justification for having the micro-drift $\mu$ in the dynamics \eqref{eq.mainDyn.def}. In addition, the values of $\hat\alpha\hat\beta$, plotted on Figure \ref{fig:alpha.beta} show that this quantity (which determines the strength of the micro-drift in $X$) remains non-negligible (to wit, $\alpha\beta=1$ means that the force created by the micro-drift, in the absence of any other forces, would move the latent price by around $\$0.0025$ per second), which also supports the presence of micro-drift in \eqref{eq.mainDyn.def}.

\section{Appendix}

\emph{Proof of Proposition \ref{prop:Sigma}}.
Throughout this proof, we fix $\bar\PP^\theta$ and assume that all expectations and ``a.s." statements hold w.r.t. this measure even though it is not explicitly cited.
Consider the two auxiliary continuous-time discrete-space processes $\tilde X$ and $\hat X$, defined recursively as
\begin{align*}
&\tilde X_t := X_{\tau(t)},\quad \tau(t):=\max_{\tau_i\leq t}\tau_i,
\quad\tau_i := \inf\{t\geq \tau_{i-1}:\, X_t\in\ZZ\setminus\{X_{\tau_{i-1}}\}\},
\quad i\geq1,\quad\tau_0=0,\\
&\hat X_t = \frac{1}{2}(B(X_t)+A(X_t)),
\end{align*}
where the functions $B$ and $A$ are defined in \eqref{eq.B.def}--\eqref{eq.A.def}.
Note that
\begin{align*}
|\tilde X_t - \hat X_t| \leq 1.
\end{align*}

\medskip

Next, we notice that
\begin{align*}
\frac{1}{t} (\tilde X_t - \tilde X_0)^2 = \frac{1}{t} \left(\xi_1 + \sum_{i=1}^{N_t} \xi_{i+1}\right)^2,\quad \xi_i:= \tilde X_{\tau_i} - \tilde X_{\tau_{i-1}},
\end{align*}
where $N_t$ is the largest $i\geq0$ s.t. $\tau_{i+1}\leq t$ (we set $N_t=0$ if there is no such $i$).
Notice also that $\{\xi_i\}$ are i.i.d. taking values $\pm1$ with probability $1/2$.
The reflection principle applied to $X$ on $[\tau_i,\tau_{i+1}]$ yields that $\{\xi_i\}_{i=2}^\infty$ is independent of $(N,\xi_1)$.\footnote{The argument goes as follows. For a given realization of the process $X$, we modify this process on the interval $[\tau_i,\tau_{i+1}]$, recursively over $i=1,2,\ldots$, so that the path $X-X_{\tau_i}$ remains the same on this interval with probability $1/2$ or becomes equal to $X_{\tau_i}-X$ with probability $1/2$, independent of everything else (the paths on the sub-intervals are then glued together so that the new process is continuous). Then, the symmetry of the drift $\mu$ yields that the new process satisfies the same SDE. The uniqueness of the solution to this SDE in law yields that the new process has the same distribution as $X$. As the joint distribution of $N_t$ and $\{\xi_i\}_{i=2}^\infty)$ depends only on the distribution of $X$, and the two are clearly independent for the new process, we obtain the desired conclusion.}
In addition, using the strong Markov property of $X$, we conclude that $\{\tau_{i+1}-\tau_i\}_{i=1}^\infty$ are i.i.d.. It is also easy to see that $\lim_{t\rightarrow\infty}N_t=\infty$ a.s.. Then, as $t\rightarrow\infty$, we have, a.s., that
$$
\frac{t}{N_t} = \frac{t-\tau_{N_t} + \tau_1}{N_t} +\frac{1}{N_t} \sum_{i=1}^{N_t} (\tau_{i+1}-\tau_{i})
\rightarrow \EE (\tau_2 - \tau_{1}),
$$
due to the strong law of large numbers. In the above, we also used the fact that 
$$
\frac{t-\tau_{N_t}}{N_t}\leq \frac{\tau_{N_t+1}-\tau_{N_t}-\EE (\tau_{2}-\tau_1)}{N_t} + \frac{\EE (\tau_{2}-\tau_1)}{N_t}
=: M_{N_t} + \frac{\EE (\tau_{2}-\tau_1)}{N_t},
$$
and that $(M_n)$ converges to zero a.s., as $n\rightarrow\infty$ (which can be shown via Borel-Cantelli, using $\EE (\tau_{n+1}-\tau_n)=\EE (\tau_2-\tau_1)<\infty$).
Thus, for any bounded continuous function $\phi$, by conditioning on $(N,\xi_1)$ and applying the central limit theorem, we obtain
$$
\lim_{t\rightarrow\infty}\EE \phi\left( \frac{1}{t} (\tilde X_t - \tilde X_0)^2\right) = \EE \phi\left(\left(\EE (\tau_2 - \tau_{1})\right)^{-1}  \xi^2 \right),
\quad \xi\sim N(0,1).
$$

Next, we use Taylor's formula and Fubini's theorem to deduce that, for any $\delta>0$ there exists $\varepsilon>0$ s.t.
$$
\EE \exp\left[-\varepsilon(\tau_2-\tau_1 - \EE(\tau_2-\tau_1))\right] < \exp(\varepsilon \delta).
$$
Using this observation, along with $\EE (\tau_2 - \tau_{1})>0$ and Chebyshev's inequality, we obtain:
\begin{align}
&\EE \frac{1}{t^2} (\tilde X_t - \tilde X_0)^4
= \frac{1}{t^2} \left(\EE \xi^4_1 + 6\EE \xi^2_1\left(\sum_{i=1}^{N_t} \xi_{i+1}\right)^2 + \EE \left(\sum_{i=1}^{N_t} \xi_{i+1}\right)^4\right)\nonumber\\
&= \frac{1}{t^2} \left(1 + 6\EE N_t + \EE N^2_t\right)
\leq \frac{C}{t^2} (1+\EE N^2_t) = \frac{C}{t^2} \left(1 + \sum_{i=2}^{\infty}i\PP(\tau_i\leq t)\right)\nonumber\\
&= \frac{C}{t^2} \left(1 + \EE\sum_{i=2}^{\infty}i\PP\left[-\left(\tau_i-\tau_1-i\EE(\tau_2-\tau_1)\right)\geq i\EE(\tau_2-\tau_1)-t+\tau_1\,\vert\,\tau_1\right]\right)\nonumber\\
&\leq \frac{C}{t^2} \left(1 + Ct^2 + \sum_{i\geq 2t/\EE(\tau_2-\tau_1)} i\PP\left[-\left(\tau_i-\tau_1-i\EE(\tau_2-\tau_1)\right)\geq i\EE(\tau_2-\tau_1)/2\right]\right)\label{eq.MME.tildeX.4.est}\\
&\leq \frac{C}{t^2} \left(1 + Ct^2 + C\sum_{i=1}^{\infty} i\EE\exp\left[-\varepsilon(\tau_i-\tau_1-i\EE(\tau_2-\tau_1))\right]
\exp\left[-i\varepsilon\EE(\tau_2-\tau_1)/2\right] \right)\nonumber\\
&= \frac{C}{t^2} \left(1 + Ct^2 + C\sum_{i=1}^{\infty} i\left[\EE\exp\left(-\varepsilon\left(\tau_2-\tau_1-\EE(\tau_2-\tau_1)\right)\right)
/\exp\left(\varepsilon\EE\left(\tau_2-\tau_1\right)/2\right)\right]^i \right),\nonumber
\end{align}
and notice that the right hand side of the above is bounded uniformly over $t\geq1$, provided $\varepsilon>0$ is sufficiently small.
This observation allows us to use $\phi(x)=x$:
\begin{equation}\label{eq.MME.tildeX.2.conv.const}
\lim_{t\rightarrow\infty}\EE \frac{1}{t} (\tilde X_t - \tilde X_0)^2 = 1/\EE (\tau_2 - \tau_{1}).
\end{equation}

\smallskip

Next, we notice that the estimate \eqref{eq.MME.tildeX.4.est} is uniform over all admissible $\tilde X_0=X_0$.
In addition, the strong Markov property of $X$ allows us to apply \eqref{eq.MME.tildeX.4.est} to $\EE(\tilde X_{\tau_1+t} - \tilde X_{\tau_1})^4/t^2$ and, in particular, to obtain
$$
\lim_{t\rightarrow\infty}\EE \frac{1}{t} (\tilde X_{\tau_1+t} - \tilde X_{\tau_1})^2 = 1/\EE (\tau_2 - \tau_{1}),
$$
where the rate of the latter convergence does not depend on $X_0$.
Let us show that the rate of convergence in \eqref{eq.MME.tildeX.2.conv.const} can also be estimated uniformly over all admissible $X_0$.
To this end, we observe:
\begin{align*}
&\left|\EE \frac{1}{t} \left((\tilde X_t - \tilde X_0)^2-(\tilde X_{\tau_1+t} - \tilde X_{\tau_1})^2\right)\right|\\
&\leq \frac{1}{t}\EE \left(|\tilde X_{\tau_1+t} - \tilde X_t| + |\tilde X_{\tau_1}-\tilde X_0|\right)
\left(|\tilde X_t - \tilde X_0| + |\tilde X_{\tau_1+t} - \tilde X_{\tau_1}|\right)\\
&\leq \frac{C}{\sqrt{t}}\left[\EE \left(1+ (X_{\tau_1+t} - X_t)^2 + (X_{\tau_1}-X_0)^2\right)\right]^{1/2}
\left[\left(\left[\EE\frac{1}{t^2}(\tilde X_t - \tilde X_0)^4\right]^{1/2} + \left[\EE\frac{1}{t^2}(\tilde X_{\tau_1+t} - \tilde X_{\tau_1})^4\right]^{1/2} \right)\right]^{1/2}\\
&\leq \frac{C}{\sqrt{t}}\left[1+ \EE(X_{\tau_1+t} - X_t)^2 + \EE(X_{\tau_1}-X_0)^2\right]^{1/2}.
\end{align*}
To estimate the second expectation in the right hand side of the above, we apply It\^o's formula to $(X_{t\wedge\tau_1} - X_0)^2$ and obtain
\begin{align*}
& \EE(X_{t\wedge\tau_1} - X_0)^2 = \EE \int_0^{t\wedge\tau_1} \left(\alpha^2\bar\sigma^2  + \sigma^2
+ 2 (X_{s} - X_0)\mu(X_s)\right) ds
\leq \left(\alpha^2\bar\sigma^2  + \sigma^2 + 1\right) \EE (t\wedge\tau_1).
\end{align*}
Considering $t\rightarrow\infty$, using the dominated convergence in the left hand side of the above, and the monotone convergence in the right hand side, and applying the analogous estimate to $\EE(X_{\tau_1+t} - X_{t})^2$, we conclude
\begin{align*}
&\left|\EE \left(\frac{1}{t} \left((\tilde X_t - \tilde X_0)^2-(\tilde X_{\tau_1+t} - \tilde X_{\tau_1})^2\right)\right)\right|
\leq \frac{C}{\sqrt{t}}\left[1+ \EE \tau_1\right]^{1/2}.
\end{align*}
As the conditional expectation $\EE (\tau_1\,\vert\,X_0=x)$ is computed explicitly below, we easily deduce that the right hand side of the above is bounded uniformly over all admissible $X_0$, which in turn yields the desired uniformity of the convergence in \eqref{eq.MME.tildeX.2.conv.const}.

\smallskip

Next, we notice that, as $t\rightarrow\infty$,
\begin{align*}
&\left|\EE \left(\frac{1}{t} (\tilde X_t - \tilde X_0)^2\right) - \EE \left(\frac{1}{t} (\hat X_t - \hat X_0)^2\right) \right|
\leq \frac{2}{t} \EE (|\tilde X_t - \tilde X_0| + |\hat X_t - \hat X_0|)\\
&\leq \frac{4}{t} \EE |X_t-X_0| + O(1/t) 
\leq 4 \left[\EE \left(\frac{1}{t}\int_0^t \mu(X_s) ds - \EE \frac{1}{t}\int_0^t \mu(X_s) ds \right)^2\right]^{1/2}\\
& + 4\left| \EE \frac{1}{t}\int_0^t \mu(X_s) ds \right| + O(1/\sqrt{t})
\rightarrow 0,
\end{align*}
due to the ergodicity of $X\,\text{mod}\,1$ (see Proposition \ref{le:ergod} and Step 2 of its proof), the dominated convergence theorem, and the observation $\int_0^1 \mu(x) \chi(x) dx = 0$. 
Notice also that the above convergence is uniform over all distributions of $X_0$ (see Proposition \ref{le:ergod} and Step 2 of its proof).
Thus, we conclude that
\begin{equation}\label{eq.MME.hatX.2.conv.const}
\lim_{t\rightarrow\infty}\EE \left(\frac{1}{t} (\hat X_t - \hat X_0)^2\right) = 1/\EE (\tau_2 - \tau_{1}),
\end{equation}
uniformly over all $X_0$.

\medskip

Let us now complete the proof of Proposition \ref{prop:Sigma}, by establishing
\begin{equation}\label{eq.sepZ.MainEq}
\lim_{T\rightarrow\infty}\frac{1}{T} \sum_{k=1}^{M(T)} \left(\hat X_{kT/M(T)} - \hat X_{(k-1)T/M(T)}\right)^2
= 1/\EE (\tau_i - \tau_{i-1}),
\end{equation}
and showing that the right hand side of the above equals $\Sigma(\theta)$.

To show that the convergence \eqref{eq.sepZ.MainEq} holds in probability, we proceed as follows:
\begin{align}
&\EE\left( \frac{1}{M} \sum_{k=1}^{M} \left[\frac{M}{T}\left(\hat X_{kT/M} - \hat X_{(k-1)T/M}\right)^2 - \EE_{X_{(k-1)T/M}} \frac{M}{T}\left(\hat X_{kT/M} - \hat X_{(k-1)T/M}\right)^2\right]\right)^2\nonumber\\
&=\frac{1}{T^2} \sum_{k,j=1}^M \EE\left(\left[\left(\hat X_{kT/M} - \hat X_{(k-1)T/M}\right)^2 - \EE_{X_{(k-1)T/M}} \left(\hat X_{kT/M} - \hat X_{(k-1)T/M}\right)^2\right]\cdot\right.\nonumber\\
&\left.\phantom{?????????????????????}
\left[\left(\hat X_{jT/M} - \hat X_{(j-1)T/M}\right)^2 - \EE_{X_{(j-1)T/M}} \left(\hat X_{jT/M} - \hat X_{(j-1)T/M}\right)^2\right]\right)\nonumber\\
&=\frac{1}{M^2} \sum_{k=1}^M \frac{M^2}{T^2} \EE\left[\left(\hat X_{kT/M} - \hat X_{(k-1)T/M}\right)^2 - \EE_{X_{(k-1)T/M}} \left(\hat X_{kT/M} - \hat X_{(k-1)T/M}\right)^2\right]^2,\label{eq.MME.MC.est.varEst}
\end{align}
where we eliminated the terms with $k\neq j$ by using the tower property of conditional expectation.
To estimate the remaining terms, we notice that 
\begin{align*}
&\EE \frac{1}{t^2} (\hat X_t - \hat X_0)^4
\leq  C \EE \frac{1}{t^2} (1 + (\tilde X_t - \tilde X_0)^4),
\end{align*}
and the right hand side of the above is estimated in \eqref{eq.MME.tildeX.4.est} uniformly over $t\geq 1$ and over all admissible $X_0$. This implies that the right hand side of \eqref{eq.MME.MC.est.varEst} converges to zero as $M\rightarrow\infty$.
It only remains to apply \eqref{eq.MME.hatX.2.conv.const}, to conclude that, a.s.,
\begin{align*}
&\frac{1}{M} \sum_{k=1}^{M} \EE_{X_{(k-1)T/M}} \frac{M}{T}\left(\hat X_{kT/M} - \hat X_{(k-1)T/M}\right)^2
\rightarrow 1/\EE (\tau_2 - \tau_{1}),
\end{align*}
as $T/M\rightarrow\infty$, which yields \eqref{eq.sepZ.MainEq}.

It only remains to compute $1/\EE (\tau_2 - \tau_{1})$.
Denote by $\iota$ the first exit time of $X$ from the interval $(-1,1)$ and consider $u(x):=\EE_x\iota$ for $x\in[-1,1]$. Then, it suffices to compute $u(0)=\EE_0\tau_1=\EE (\tau_2 - \tau_{1})$. The Feynman-Kac formula implies that $u$ satisfies
\begin{equation*}
\frac{1}{2}(\alpha^2\bar\sigma^2 + \sigma^2) u_{xx} +\alpha\beta \mu(x) u_x = -1,\quad x\in(-1,1),
\end{equation*}
subject to $u(-1)=u(1)=0$. Solving this equation explicitly, we obtain
\begin{align*}
& u(x) = - \frac{2}{\sigma^2}\int_{-1}^x\exp\left[-\frac{2\alpha\beta}{\sigma^2} \int_0^z \mu(r) dr\right] \int_0^z \exp\left[\frac{2\alpha\beta}{\sigma^2} \int_0^y \mu(r) dr\right] dy dz,\\
& u(0) = \frac{2}{\sigma^2} \int_{0}^1\exp\left[-\frac{\alpha\beta}{\sigma^2} (z-1/2)^2\right] \int_0^z \exp\left[\frac{\alpha\beta}{\sigma^2} (y-1/2)^2\right] dy dz
= 1/\Sigma(\theta).
\end{align*}
\qed







	
	


\begin{figure}
\begin{center}
    \includegraphics[width = 0.5\textwidth]{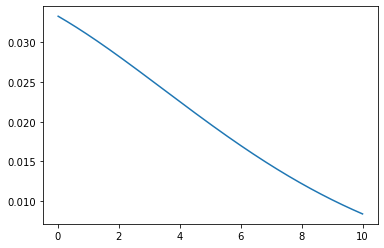}
  \caption{Graph of the function \eqref{eq.TheFunc}.}
    \label{fig:1}
\end{center}
\end{figure}

\begin{figure}
\begin{center}
  \begin{tabular} {cc}
    {
    \includegraphics[width = 0.45\textwidth]{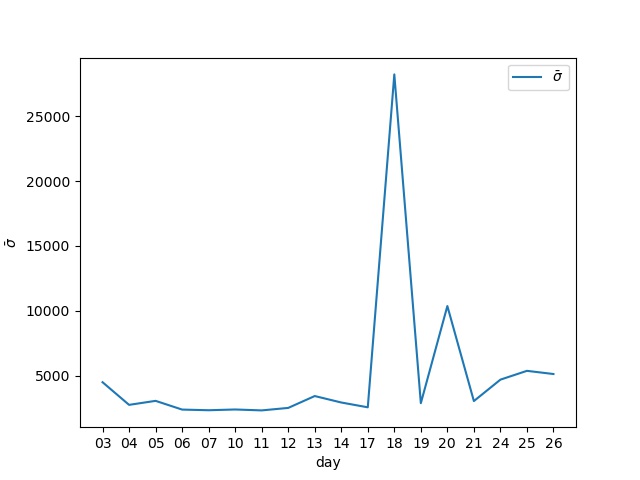}
    } & {
   \includegraphics[width = 0.45\textwidth]{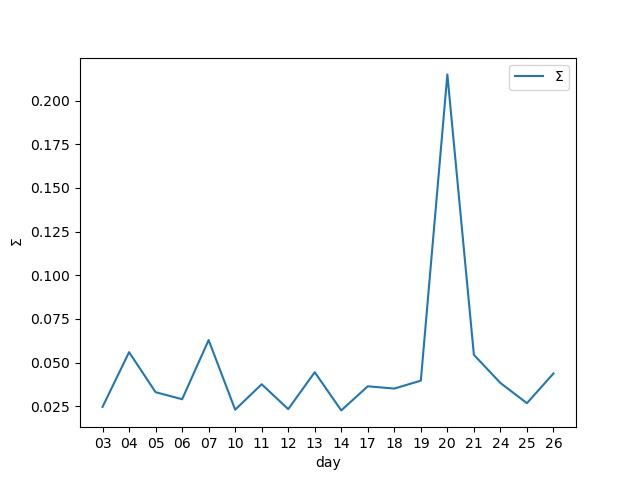}
    }\\
  \end{tabular}
  \caption{Estimated $\bar\sigma^2$ and $\hat\Sigma$ for INTC, across the estimation days.}
    \label{fig:Sigma.INTC}
\end{center}
\end{figure}

\begin{figure}
\begin{center}
  \begin{tabular} {cc}
    {
    \includegraphics[width = 0.45\textwidth]{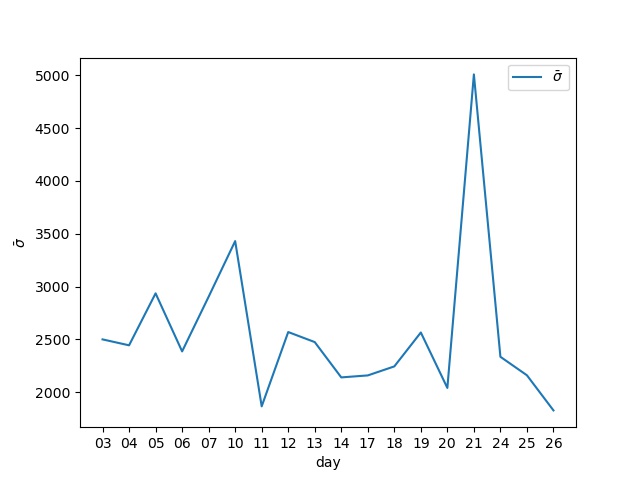}
    } & {
   \includegraphics[width = 0.45\textwidth]{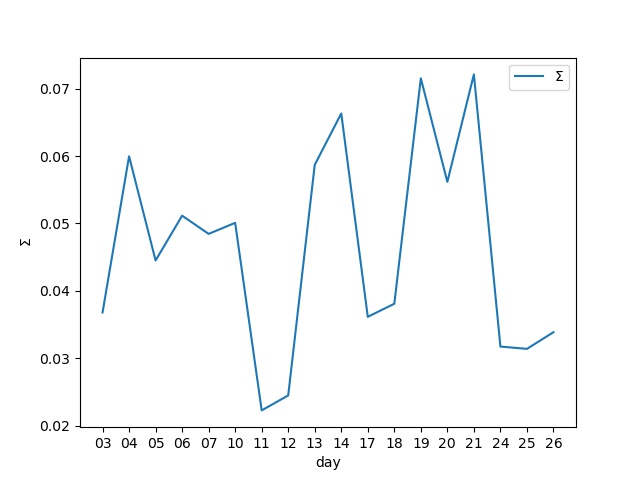}
    }\\
  \end{tabular}
  \caption{Estimated $\bar\sigma^2$ and $\hat\Sigma$ for MSFT, across the estimation days.}
    \label{fig:Sigma.MSFT}
\end{center}
\end{figure}

\begin{figure}
\begin{center}
  \begin{tabular} {cc}
    {
    \includegraphics[width = 0.45\textwidth]{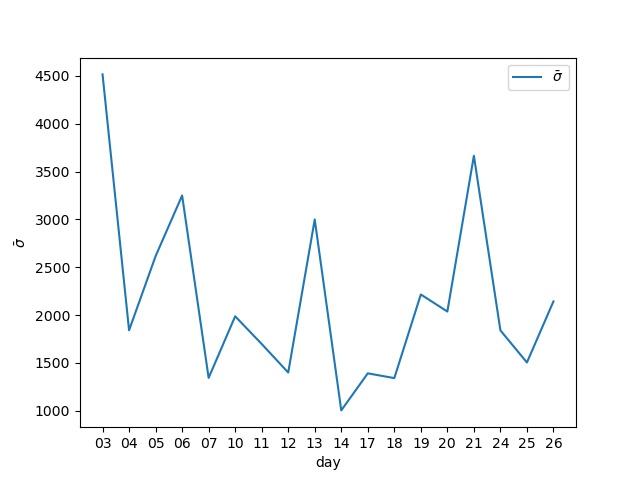}
    } & {
   \includegraphics[width = 0.45\textwidth]{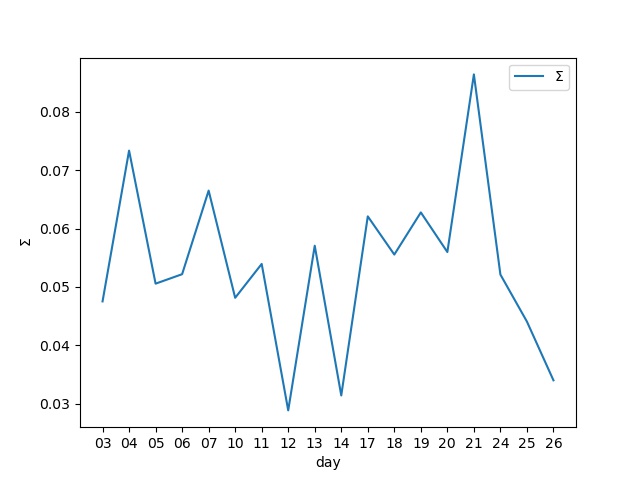}
    }\\
  \end{tabular}
  \caption{Estimated $\bar\sigma^2$ and $\hat\Sigma$ for MU, across the estimation days.}
    \label{fig:Sigma.MU}
\end{center}
\end{figure}

\begin{figure}
\begin{center}
  \begin{tabular} {cc}
    {
    \includegraphics[width = 0.45\textwidth]{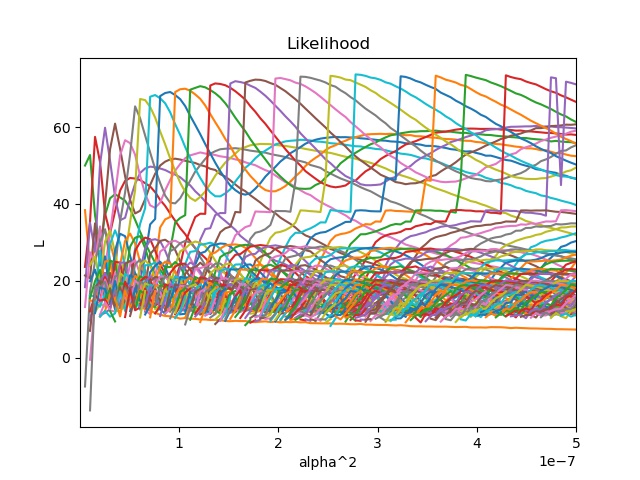}
    } & {
   \includegraphics[width = 0.45\textwidth]{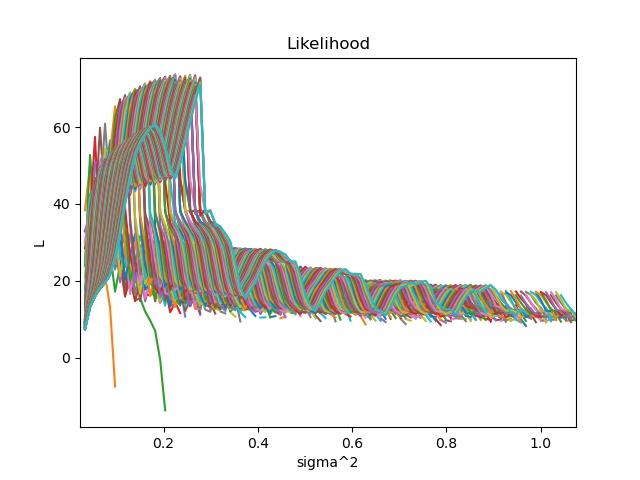}
    }\\
  \end{tabular}
  \caption{Cross-sectional plots of $\frac{1}{T}\log L_T(\cdot)$ for INTC computed on Nov 3, 2014. Left: $\frac{1}{T}\log L_T(\cdot)$ as a function of $\alpha^2$ (for each fixed value of $\sigma^2$). Right: $\frac{1}{T}\log L_T(\cdot)$ as a function of $\sigma^2$ (for each fixed value of $\alpha^2$).}
    \label{fig:L.INTC}
\end{center}
\end{figure}

\begin{figure}
\begin{center}
  \begin{tabular} {cc}
    {
    \includegraphics[width = 0.45\textwidth]{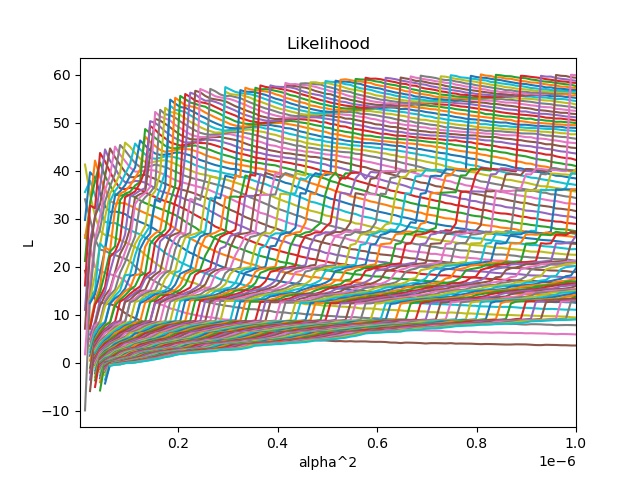}
    } & {
   \includegraphics[width = 0.45\textwidth]{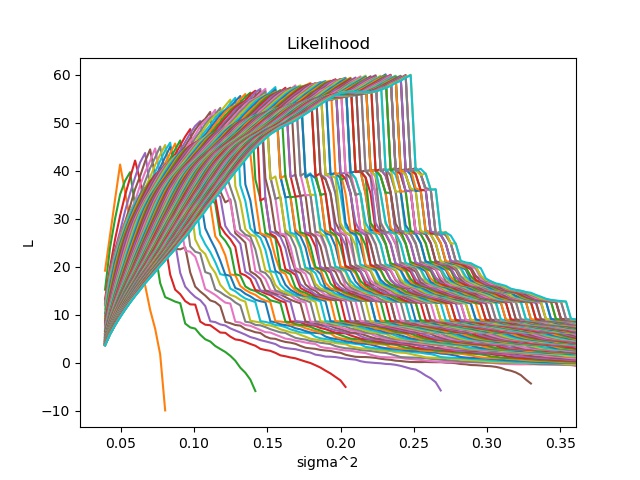}
    }\\
  \end{tabular}
  \caption{Cross-sectional plots of $\frac{1}{T}\log L_T(\cdot)$ for MSFT computed on Nov 3, 2014. Left: $\frac{1}{T}\log L_T(\cdot)$ as a function of $\alpha^2$ (for each fixed value of $\sigma^2$). Right: $\frac{1}{T}\log L_T(\cdot)$ as a function of $\sigma^2$ (for each fixed value of $\alpha^2$).}
    \label{fig:L.MSFT}
\end{center}
\end{figure}

\begin{figure}
\begin{center}
  \begin{tabular} {cc}
    {
    \includegraphics[width = 0.45\textwidth]{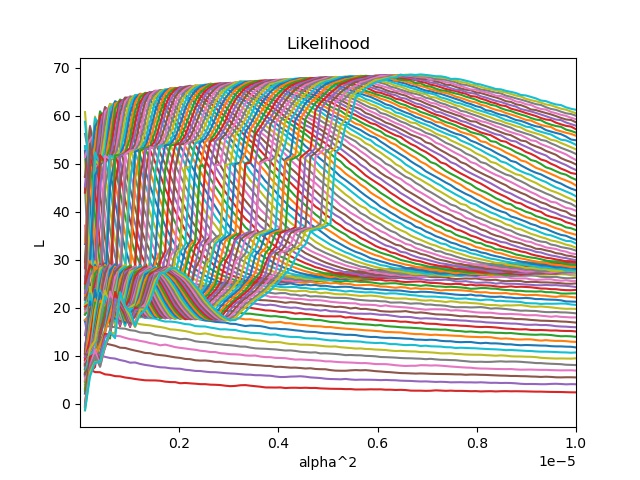}
    } & {
   \includegraphics[width = 0.45\textwidth]{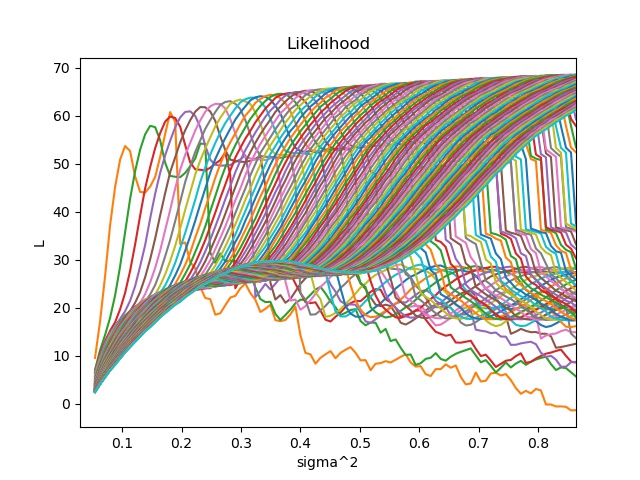}
    }\\
  \end{tabular}
  \caption{Cross-sectional plots of $\frac{1}{T}\log L_T(\cdot)$ for MU computed on Nov 3, 2014. Left: $\frac{1}{T}\log L_T(\cdot)$ as a function of $\alpha^2$ (for each fixed value of $\sigma^2$). Right: $\frac{1}{T}\log L_T(\cdot)$ as a function of $\sigma^2$ (for each fixed value of $\alpha^2$).}
    \label{fig:L.MU}
\end{center}
\end{figure}

\begin{figure}
\begin{center}
  \begin{tabular} {cc}
    {
    \includegraphics[width = 0.45\textwidth]{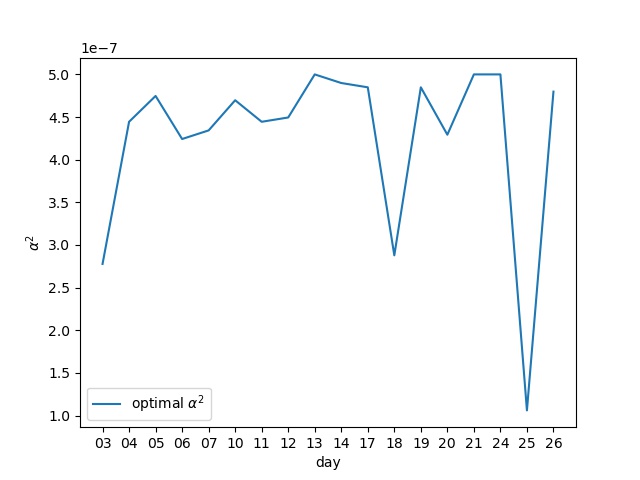}
    } & {
   \includegraphics[width = 0.45\textwidth]{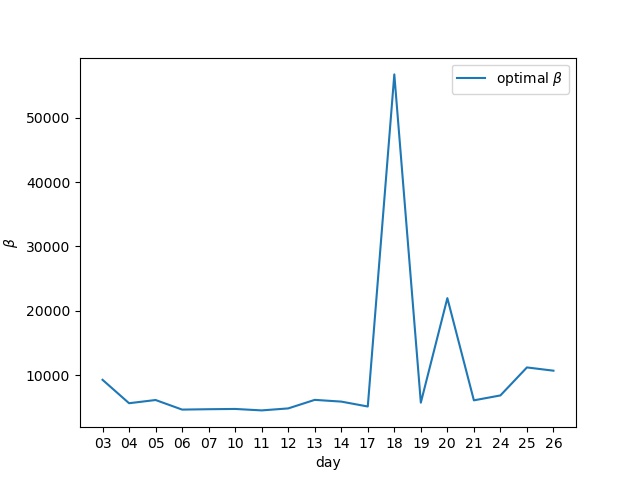}
    }\\
    {
    \includegraphics[width = 0.45\textwidth]{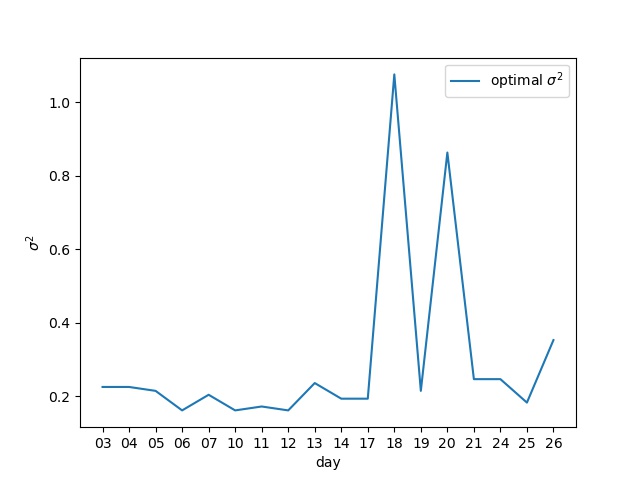}
    } & {}
  \end{tabular}
  \caption{MLE of $(\alpha^2,\beta,\sigma^2)$ for INTC, across the estimation days.}
    \label{fig:MLE.INTC}
\end{center}
\end{figure}

\begin{figure}
\begin{center}
  \begin{tabular} {cc}
    {
    \includegraphics[width = 0.45\textwidth]{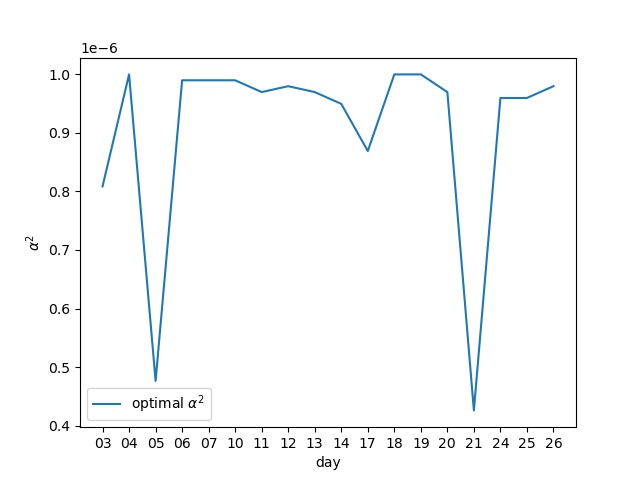}
    } & {
   \includegraphics[width = 0.45\textwidth]{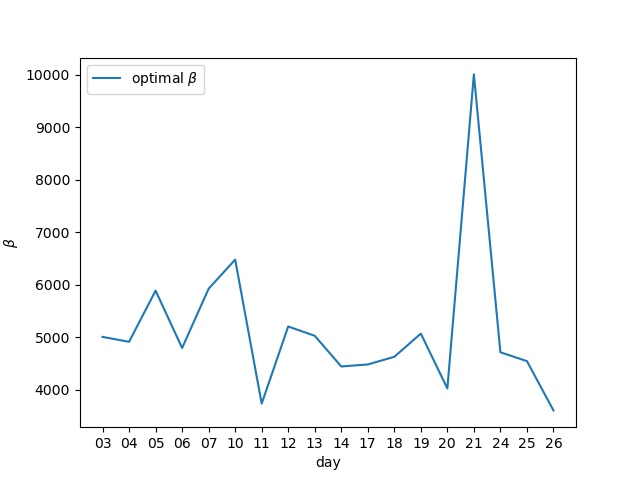}
    }\\
    {
    \includegraphics[width = 0.45\textwidth]{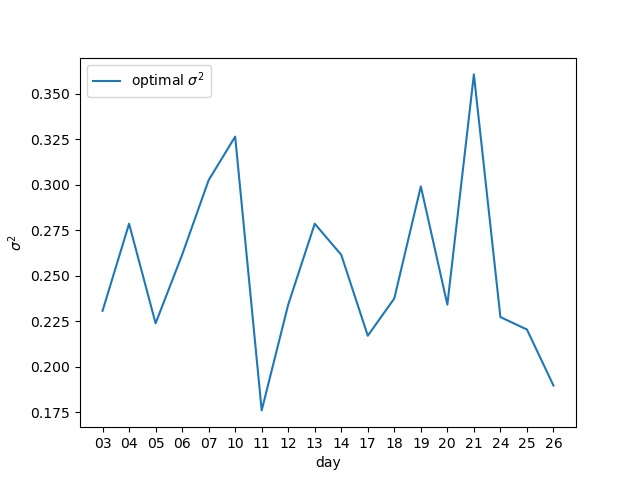}
    } & {}
  \end{tabular}
  \caption{MLE of $(\alpha^2,\beta,\sigma^2)$ for MSFT, across the estimation days.}
    \label{fig:MLE.MSFT}
\end{center}
\end{figure}

\begin{figure}
\begin{center}
  \begin{tabular} {cc}
    {
    \includegraphics[width = 0.45\textwidth]{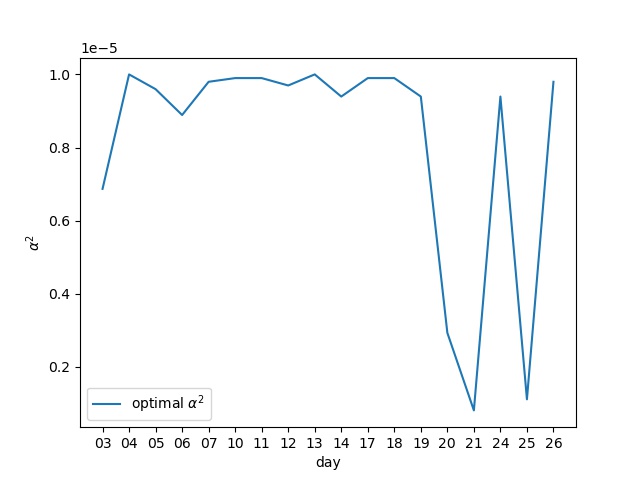}
    } & {
   \includegraphics[width = 0.45\textwidth]{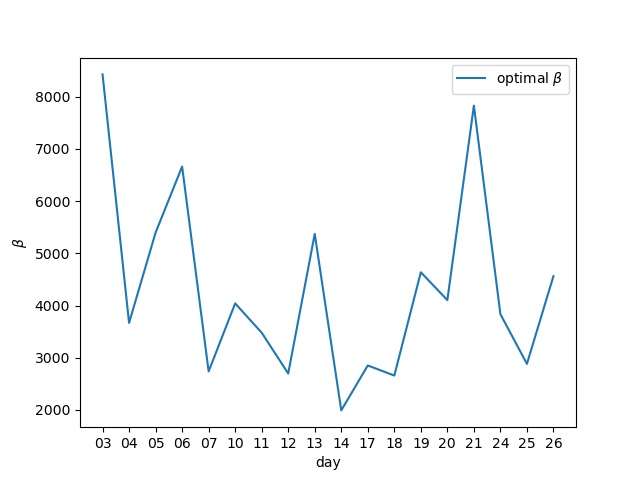}
    }\\
    {
    \includegraphics[width = 0.45\textwidth]{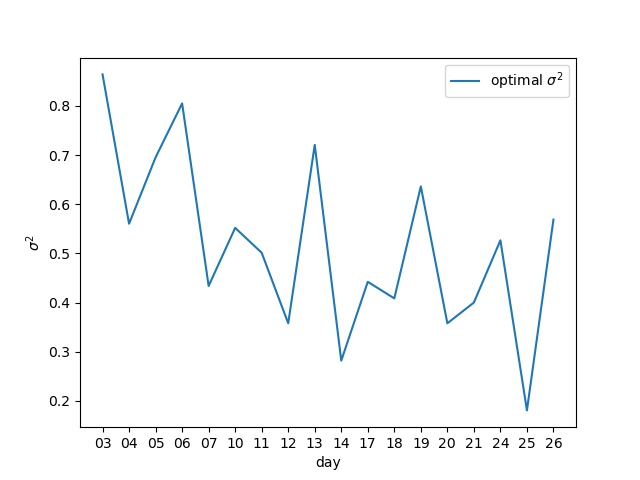}
    } & {}
  \end{tabular}
  \caption{MLE of $(\alpha^2,\beta,\sigma^2)$ for MU, across the estimation days.}
    \label{fig:MLE.MU}
\end{center}
\end{figure}

\begin{figure}
\begin{center}
  \begin{tabular} {cc}
    {
    \includegraphics[width = 0.45\textwidth]{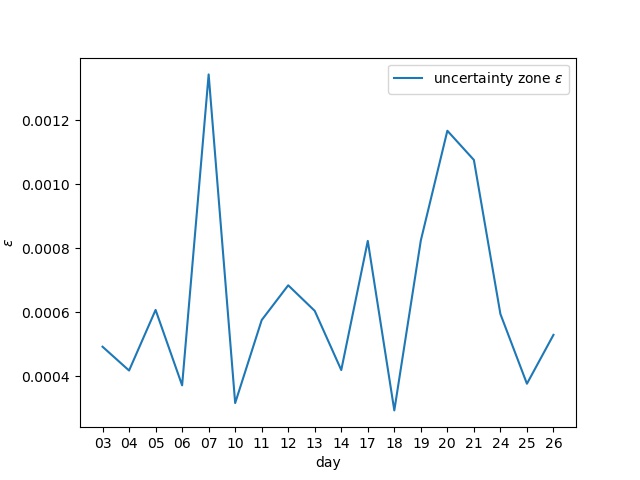}
    } & {
   \includegraphics[width = 0.45\textwidth]{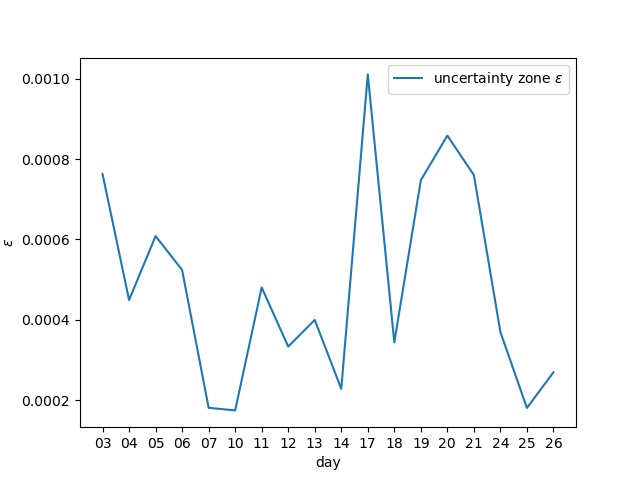}
    }\\
    {
    \includegraphics[width = 0.45\textwidth]{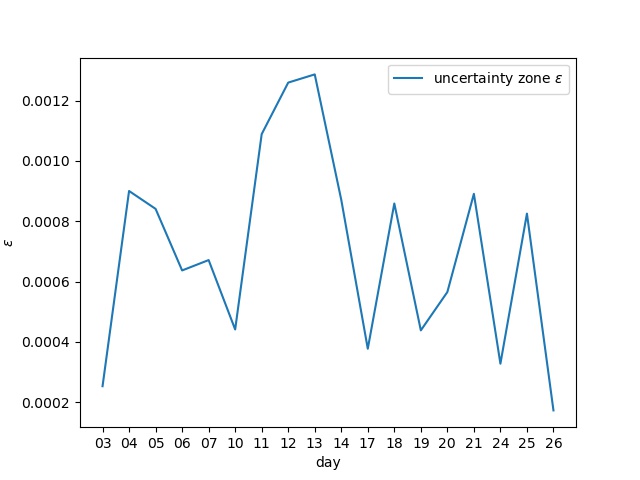}
    } & {}
  \end{tabular}
  \caption{Estimated values of $\epsilon$ for INTC (upper-left), MSFT (upper-right) and MU (lower-left), across the estimation days.}
    \label{fig:epsilon}
\end{center}
\end{figure}

\begin{figure}
\begin{center}
  \begin{tabular} {cc}
    {
    \includegraphics[width = 0.45\textwidth]{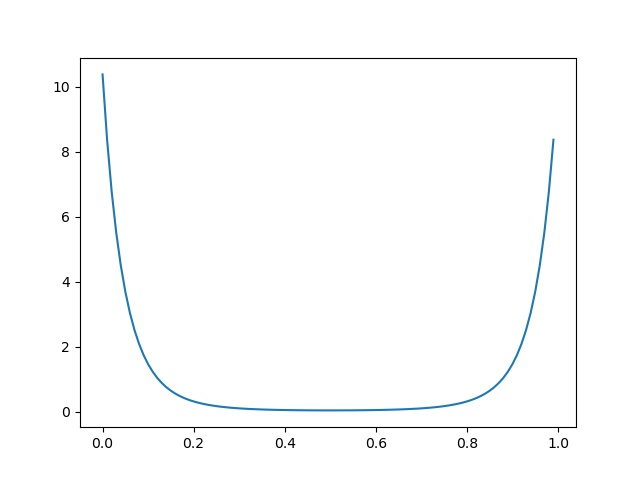}
    } & {
   \includegraphics[width = 0.45\textwidth]{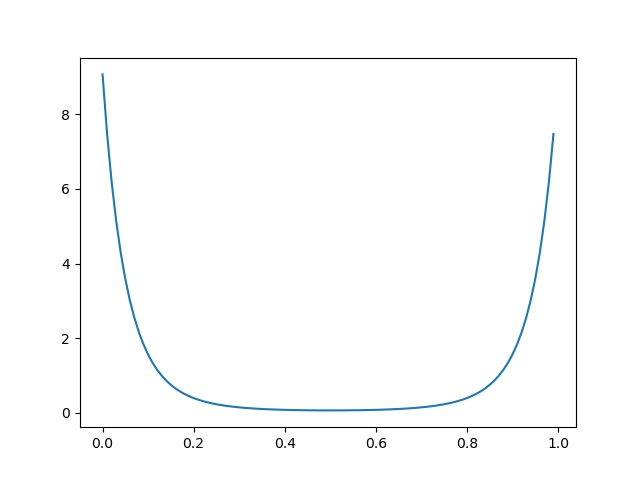}
    }\\
    {
    \includegraphics[width = 0.45\textwidth]{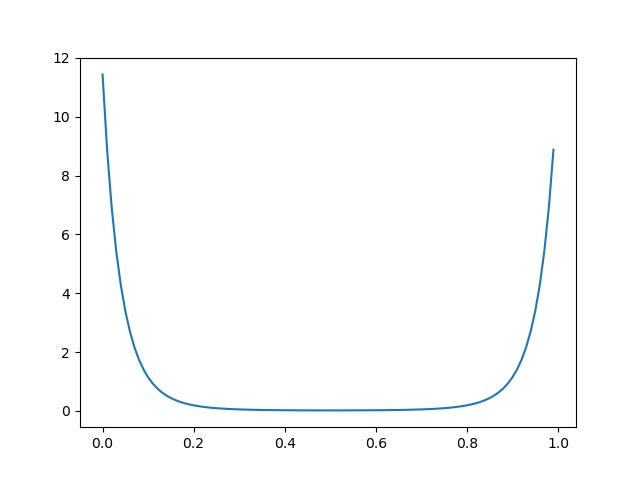}
    } & {}
  \end{tabular}
  \caption{The plots of the stationary density $\chi$ corresponding to the parameters estimated on Nov 3, 2014, for INTC (upper-left), MSFT (upper-right) and MU (lower-left).}
    \label{fig:chi}
\end{center}
\end{figure}

\begin{figure}
\begin{center}
  \begin{tabular} {cc}
    {
    \includegraphics[width = 0.45\textwidth]{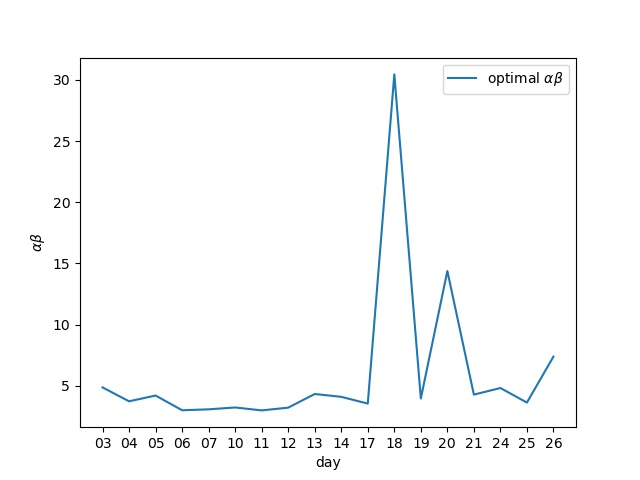}
    } & {
   \includegraphics[width = 0.45\textwidth]{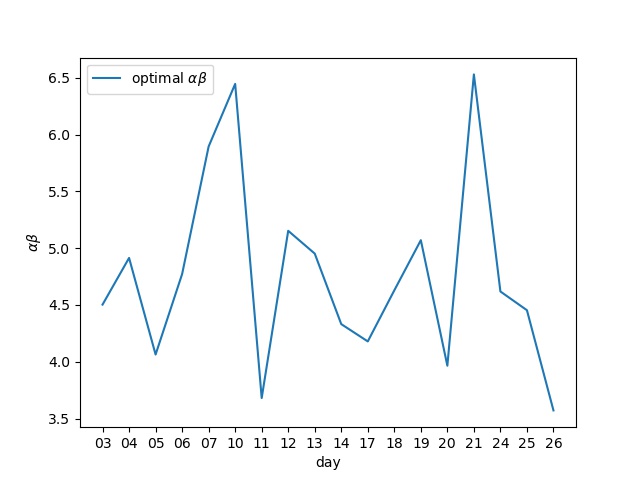}
    }\\
    {
    \includegraphics[width = 0.45\textwidth]{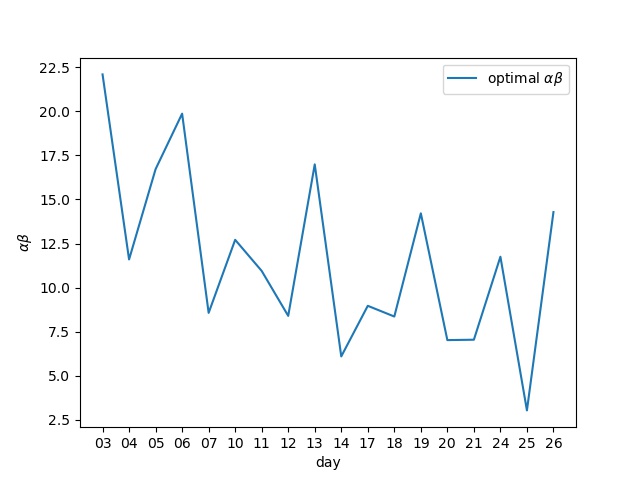}
    } & {}
  \end{tabular}
  \caption{The plots of estimated $\alpha\beta$ for INTC (upper-left), MSFT (upper-right) and MU (lower-left), across the estimation days.}
    \label{fig:alpha.beta}
\end{center}
\end{figure}

\bibliographystyle{amsalpha}
\bibliography{MLE}

\end{document}